\documentclass[paper=A4, DIV=12, USenglish, numbers=noenddot]{scrartcl} % USenglish for autoref
\usepackage{cmap}		% to search and copy ligatures
\usepackage[utf8]{inputenc}	% for Linux computer and Mac
\usepackage[T1]{fontenc}	% to search for ligatures in the pdf
\usepackage[USenglish]{babel}
\usepackage[]{microtype}
\usepackage{booktabs} % for better tables

\usepackage{amsmath}
\usepackage{amssymb}
\usepackage{mathtools}		% for \coloneqq
\usepackage{xfrac}

\usepackage{aliascnt}		% for Aliascounter so that autoref gives the thms the right names
\usepackage{amsthm}

\usepackage{tikz}		% for diagrams and other sketches
\usetikzlibrary{matrix,arrows,patterns,intersections,calc,decorations.pathmorphing,decorations.markings}

\usepackage[arrow, matrix, curve]{xy}

\usepackage[pdftitle={Veech groups of covers of the Chamanara surface},pdfauthor={Artigiani, Randecker, Sadanand, Valdez, Weitze-Schmithüsen},pdfborder={0 0 0}]{hyperref}
\usepackage[figure, table]{hypcap}	% to make autoref link to figures and not the captions of figures

\newtheoremstyle{break}
 {} %space above
 {} %space below
 {\itshape} %body font
 {} %indent amount
 {\bfseries} %theorem head font
 {} %punctuation
 {\newline} %space after theorem head
 {\thmname{#1}\thmnumber{ #2}\thmnote{ (#3)}} %theorem head spec

\newtheoremstyle{breakdef}
 {} %space above
 {} %space below
 {} %body font
 {} %indent amount
 {\bfseries} %theorem head font
 {} %punctuation
 {\newline} %space after theorem head
 {\thmname{#1}\thmnumber{ #2}\thmnote{ (#3)}} %theorem head spec

\newtheoremstyle{remark}
 {} %space above
 {} %space below
 {} %body font
 {} %indent amount
 {\itshape} %theorem head font
 {.} %punctuation
 {0.5em} %space after theorem head
 {\thmname{#1}\thmnumber{ #2}\thmnote{ {\normalfont (}#3{\normalfont )}}} %theorem head spec

\theoremstyle{breakdef}
\newtheorem{definition}{Definition}[section]

\theoremstyle{remark}

\newaliascnt{rem}{definition}
\newtheorem{rem}[rem]{Remark}
\aliascntresetthe{rem}

\newaliascnt{exa}{definition}
\newtheorem{exa}[exa]{Example}
\aliascntresetthe{exa}

\newaliascnt{lem}{definition}
\newtheorem{lem}[lem]{Lemma}
\aliascntresetthe{lem}

\theoremstyle{break}

\newtheorem{thm}{Theorem}

\newaliascnt{prop}{definition}
\newtheorem{prop}[prop]{Proposition}
\aliascntresetthe{prop}

\newaliascnt{cor}{definition}
\newtheorem{cor}[cor]{Corollary}
\aliascntresetthe{cor}

\newtheorem*{corintro}{Corollary to Theorem 3}

\usepackage[nameinlink,capitalize,noabbrev]{cleveref}
\crefname{rem}{Remark}{Remarks}
\crefname{definition}{Definition}{Definitions}
\crefname{exa}{Example}{Examples}
\crefname{lem}{Lemma}{Lemmas}
\crefname{thm}{Theorem}{Theorems}
\crefname{prop}{Proposition}{Propositions}
\crefname{cor}{Corollary}{Corollaries}
\crefname{corintro}{Corollary to Theorem 3}{Corollaries}

\makeatletter
\newsavebox{\@brx}
\newcommand{\llangle}[1][]{\savebox{\@brx}{\(\m@th{#1\langle}\)}%
	\mathopen{\copy\@brx\kern-0.5\wd\@brx\usebox{\@brx}}}
\newcommand{\rrangle}[1][]{\savebox{\@brx}{\(\m@th{#1\rangle}\)}%
	\mathclose{\copy\@brx\kern-0.5\wd\@brx\usebox{\@brx}}}
\makeatother

\newcommand{\R}{\mathbb{R}}
\newcommand{\TheSequences}{\mathcal{B}}
\newcommand{\TheSequencesQuot}{\overline{\mathcal{B}}}
\newcommand{\into}{\hookrightarrow}

\newcommand{\intersection}{i}

\DeclareMathOperator{\Id}{Id}
\DeclareMathOperator{\SL}{SL}
\DeclareMathOperator{\PSL}{PSL}
\DeclareMathOperator{\GL}{GL}
\DeclareMathOperator{\Aut}{Aut}
\DeclareMathOperator{\Trans}{Trans}
\DeclareMathOperator{\Deck}{Deck}
\DeclareMathOperator{\Sym}{Sym}
\DeclareMathOperator{\Aff}{Aff}

\DeclareMathOperator{\Hom}{Hom}
\newcommand{\NN}{\mathbb{N}}
\newcommand{\ZZ}{\mathbb{Z}}

\newcommand{\RR}{\mathbb{R}}

\DeclareMathOperator{\NNN}{\mathcal{N}}
\DeclareMathOperator{\HHH}{\mathcal{H}}

\DeclareMathOperator{\Cov}{Cov}
\DeclareMathOperator{\TransCov}{TransCov}

\newcommand{\p}{p}

\title{Veech groups of covers \\ of the Chamanara surface}
\author{Mauro Artigiani\thanks{Universidad Nacional de Colombia - Bogot\'a} \and
		Anja Randecker\thanks{Heidelberg University, Germany} \thanks{Universit\"{a}t des Saarlandes, Germany} \and
		Chandrika Sadanand\thanks{Bowdoin College, ME, USA} \and
		Ferrán Valdez\thanks{Centro de Ciencias Matem\'aticas, UNAM Campus Morelia, M\'exico} \and
		Gabriela Weitze-Schmith\"{u}sen\footnotemark[3]}
\date{November 18, 2026}

\begin{document}

\maketitle

\begin{abstract}
	We study finite abelian covers of the Chamanara surface, an example
	of a finite-area infinite translation surface with interesting dynamics and a large Veech group.
	Specifically, the Veech group of the Chamanara surface is a virtually free group on two generators.
	We characterize when finite abelian covers have large Veech groups themselves, namely
	when their Veech group has finite index in that of the Chamanara surface.
	For degree-$2$ covers, we provide a detailed analysis of these finite-index Veech
	groups. As an application, we prove that every free group
	arises as the projective Veech group of a finite-area infinite translation surface.
\end{abstract}

A \emph{translation surface} is a Riemann surface $X$ with an atlas of charts
$\mathcal{U}$ for which all transition functions are locally translations. Translation surfaces occur naturally, via a procedure known as \emph{unfolding}, when studying the
motion of an ideal billiard ball inside a polygonal billiard table,
see~\cite{FoxKershner, KatokZemlyakov, HubertSchmidt}.

A translation surface is called \emph{finite} if its metric completion
$\overline{X}$ is a compact surface without boundary. In this case,
the \emph{set of singularities}, i.e., the set $\overline{X}\setminus X$,
consists of finitely many points which are conical singularities of cone angle
$2(k+1)\pi$, for an integer $k\geq 0$. The theory of finite translation
surfaces has been developed intensively since the pioneering work of Masur and
Veech in the~1980s. A crucial tool in the area is the moduli space of
translation surfaces of genus $g$, denoted $\mathcal{H}_g$, and the action on
this space by the group $\SL(2,\RR)$, see \cref{sec:chamanara}.

The dynamics of the geodesic flow on a given finite
translation surface is intimately connected to the dynamics
of the translation surface inside the moduli space $\mathcal{H}_g$ under $\SL(2,\RR)$. Given a
translation surface $(X, \mathcal{U})$, the stabilizer for the $\SL(2,\RR)$--action is called
the \emph{Veech group}, which we denote by $\Gamma(X)$, omitting the reference to
the translation atlas. This is a discrete, not co-compact subgroup of
$\SL(2,\RR)$, see \cite{Veech, Vorobets96}, which is trivial for generic
surfaces~\cite{Moller:affine}. The celebrated \emph{Veech dichotomy} states that
the geodesic flow on a translation surface has optimal dynamical behaviour if
$\Gamma(X)$ is a \emph{lattice}~\cite{Veech}. By a theorem of
Smillie~\cite{GekhtmanWright}, this happens if and only if the
$\SL(2,\RR)$--orbit is \emph{closed} inside $\mathcal{H}_g$. In this case, the
projection of the orbit on the moduli space of algebraic curves of genus $g$,
$\mathcal{M}_g$, under the forgetful map, yields a complex geodesic for the
Teichmüller metric, called a \emph{Teichmüller curve}. The study of
$\SL(2,\RR)$--orbit closures is a central topic in the field, and has seen
spectacular progress in the~2010s with the results of Eskin, Mirzakhani, and
Mohammadi~\cite{EM, EMM}.

Since the early 2000s, in parallel to the above theory, the study of
infinite translation surfaces, i.e., surfaces which are not finite, also began to emerge.
For instance, the surface obtained
by unfolding a polygonal billiard with at least one angle which is not a
rational multiple of $\pi$ always has infinite topological type, i.e., its fundamental group is
not finitely generated, and so it is not a finite translation surface. More
precisely, topologically we always obtain a so-called \emph{Loch Ness
monster}, which is, up to homeomorphism, the only surface with infinite genus
and only one end~\cite{Valdez}. Leaving the finite translation surface world
creates many interesting new phenomena. Infinite cone angles can appear in the
metric completion~\cite{HubertSchmithusen}, and there are also singularities which
are not of conical type, called \emph{wild singularities},
see~\cite{bowman_valdez_13,randecker_16}. Moreover, there are examples of
infinite translation surfaces, such as the \emph{icicled
surface}~\cite{randecker_16}, in which the geodesic flow is not
defined for all time for (Lebesgue) almost every point and direction. An
introduction to infinite translation surfaces can be found
in~\cite{DHV:infinite}. 

Infinite translation surfaces are in many ways harder to study than their finite
counterparts. A major challenge is the lack of a good moduli space, analogous to
$\mathcal{H}_g$; see~\cite{Hooper:moduli, Trevino} for some progress in this
direction. Perhaps surprisingly, however, Veech groups are better understood in
the \emph{infinite} case than in the finite one. In fact, only some classes of
Veech groups for finite translation surfaces are well understood,
see~\cite{schmithuesen_2005, ellenberg_mcreynolds_12, hadari_10,
thevis_22}, and for few types of groups, it is known that they cannot occur,
see \cite{KenyonSmillie00,HubertSchmidt01}. In general, it is a wide open
problem which groups can be realized as Veech groups of finite translation
surfaces. On the contrary, the groups that can appear as Veech groups of
infinite translation surfaces are completely classified
in~\cite{przytycki_schmithuesen_valdez_11}. Moreover, the realization problem
was also studied
in~\cite{ramirez_valdez_17,artigiani_randecker_sadanand_valdez_weitze-schmithuesen_23}.
In particular, using techniques of Thurston, Veech, and Hooper, one can realize
the free group on two generators in a large class of
surfaces~\cite{morales_valdez_22}. The relationship between the Veech group and the
dynamics of the geodesic flow on infinite translation surfaces, however, is less
clear. There are known examples for lattice surfaces, i.e., surfaces whose
Veech group is a lattice~\cite{Hooper:Veech1, Hooper:Veech2}, but it is not
known whether a version of the Veech dichotomy holds for infinite translation
surfaces.

\begin{figure}[t]
	\centering
	\begin{tikzpicture}[x=1cm,y=1cm, scale=0.5]
		\draw (0,0) -- node[below] {$b_1$}
		(4,0) -- node[below] {$b_2$}
		(6,0) -- node[below] {$b_3$}
		(7,0) -- node[below] {$b_4$}
		(7.5,0) -- (7.75,0);
		\draw[densely dotted] (7.75,0) -- (8,0);

		\fill (0,0) circle (2pt);
		\fill (8,8) circle (2pt);

		\fill (4,0) circle (2pt);
		\fill (6,0) circle (2pt);
		\fill (7,0) circle (2pt);
		\fill (7.5,0) circle (2pt);
		\fill (7.75,0) circle (2pt);

		\draw (8,8) -- node[above] {$b_1$}
		(4,8) -- node[above] {$b_2$}
		(2,8) -- node[above] {$b_3$}
		(1,8) -- node[above] {$b_4$}
		(0.5,8) -- (0.25,8);
		\draw[densely dotted] (0.25,8) -- (0,8);

		\fill (4,8) circle (2pt);
		\fill (2,8) circle (2pt);
		\fill (1,8) circle (2pt);
		\fill (0.5,8) circle (2pt);
		\fill (0.25,8) circle (2pt);

		\draw (0,0) -- node[left] {$a_1$}
		(0,4) -- node[left] {$a_2$}
		(0,6) -- node[left] {$a_3$}
		(0,7) -- node[left] {$a_4$}
		(0,7.5) -- (0,7.75);
		\draw[densely dotted] (0,7.75) -- (0,8);

		\fill (0,4) circle (2pt);
		\fill (0,6) circle (2pt);
		\fill (0,7) circle (2pt);
		\fill (0,7.5) circle (2pt);
		\fill (0,7.75) circle (2pt);

		\draw (8,8) -- node[right] {$a_1$}
		(8,4) -- node[right] {$a_2$}
		(8,2) -- node[right] {$a_3$}
		(8,1) -- node[right] {$a_4$}
		(8,0.5) -- (8,0.25);
		\draw[densely dotted] (8,0.25) -- (8,0);

		\fill (8,4) circle (2pt);
		\fill (8,2) circle (2pt);
		\fill (8,1) circle (2pt);
		\fill (8,0.5) circle (2pt);
		\fill (8,0.25) circle (2pt);
	\end{tikzpicture}
	\caption{The Chamanara surface, also known as the baker's map surface. We identify, using translations, segments that are parallel and have the same length.}
	\label{fig:chamanara}
\end{figure}

This article focuses on the study of Veech groups of infinite translation
surfaces of \emph{finite area}, which are much less understood, since all the
results listed above are obtained by constructions which inherently produce translation surfaces of infinite area.
Finite-area infinite translation surfaces have rich dynamics, see~\cite{LindseyTrevino, Gordillo} for some results in this direction,
and in the case that the Veech group is a lattice, it would be a natural question whether (an adapted version of) the Veech dichotomy holds.
However, it is not yet known whether there exist  infinite lattice surfaces of finite
area at all.
Finite-area infinite translation surfaces with
``large'' Veech groups, in the sense that they contain free groups on two or
more generators, are the closest to lattice surfaces as we can get right now, and hence are particularly interesting.

Historically, one of the first
studied examples of finite-area infinite translation surfaces was the
\emph{Chamanara surface}, see \cite{chamanara_04}. %, ChamanaraGardinerLakic06.
It is represented in \cref{fig:chamanara}. The
projective Veech group of the Chamanara surface, i.e., the image of its Veech group in $\PSL(2,\RR)$, is a free group on two
generators~\cite{chamanara_04,herrlich_randecker_16}. A natural question is thus
whether free groups on more than two generators can also appear as Veech groups
for finite-area infinite translation surfaces. Our first main result
gives an affirmative answer to this question for the projective Veech group.

\begin{thm}[Realizing $F_n$ as Veech group] \label{thm:realization_F_n}
	Let $n\in \mathbb{N}$. Then the free group on $n$ generators can be
	realized as the projective Veech group of an infinite translation surface of
	finite area.
\end{thm}

Moreover, we show that all the surfaces constructed
above are topologically Loch Ness monsters, see \cref{thm:topology_d=2}. We
remark that the analogue of \cref{thm:realization_F_n} for finite translation
surfaces is also true. In fact, Ellenberg and McReynolds showed in
\cite{ellenberg_mcreynolds_12} that, when P\(\Gamma(2)\) denotes the image of the
principal congruence group \(\Gamma(2)\) of level $2$ in \(\PSL(2,\ZZ)\), then
any subgroup of finite index of P$\Gamma(2)$, occurs as projective Veech group
of a finite square-tiled surface. The result then follows from the
observation that P$\Gamma(2)$ is a free group on two generators and that, by
Basse--Serre theory, any finitely generated free group occurs as finite-index
subgroup of the free group on two generators.

We prove \cref{thm:realization_F_n} by studying finite abelian covers of the
Chamanara surface. Let us stress that, while both finite- and infinite-degree
covers of finite translation surfaces have been much studied, not much is known
about covers of infinite translation surfaces. An exception is the work of
Hooper and Treviño~\cite{hooper_trevino_19}. Following their approach, in
\cref{sec:monodromy}, we explain how to identify  normal covers of the Chamanara
surface~$X$ with finite deck group~$G$ with a monodromy vector $h\in
G^\mathbb{Z}$, which we call a \emph{defining vector} of the cover. Thus, we
identify the space of normal covers with deck group $G$ with a subquotient of the
space $G^\mathbb{Z}$, which we denote~$\TheSequences_G$. Given a defining vector
$h$, we denote by $Y_h$ the corresponding cover.

A crucial ingredient is that the Chamanara surface has the
\emph{everything descends property}, i.e., for any covering $Y \to X$, every
affine homeomorphism  of $Y$ descends to $X$, see
\cref{prop:everything_descends}. In particular, for any defining vector $h$, we
have that the Veech group $\Gamma(Y_h)$ is always contained in the Veech group
of the  Chamanara surface $\Gamma(X)$.
Then, we obtain \cref{thm:realization_F_n} by studying covers of
degree $2$ of the Chamanara surface, and using the following characterization of
being a finite-index subgroup of the virtually free group $\Gamma(X)$.

\begin{thm}[Characterization of finite index ($d=2$)]\label{thm:criterion_finite_index_p=2}
	Let $h$ be a defining vector of a cover $Y_h$ of the Chamanara surface $X$
	of degree $2$. Then, the index of~$\Gamma(Y_h)$ in $\Gamma(X)$ is finite if
	and only if $h$ is periodic.
\end{thm}

We give a more technical version of \cref{thm:criterion_finite_index_p=2} in
\cref{sec:d=2} which also allows us to determine the index and hence the rank of
the free group $\Gamma(Y_h)/\{\pm \Id\}$. Moreover, in the degree--$2$
case, we describe all Veech groups $\Gamma(Y_h)$ which are of finite index in
$\Gamma(X)$ explicitly by their Schreier coset graph, see
\cref{prop:all_Schreier_coset_graphs}.

A key tool in our analysis is the action of the hyperbolic element $H = \left(\begin{smallmatrix} 2 & 0 \\
	0 & \frac{1}{2} \end{smallmatrix}\right) \in \Gamma(X)$ % in the Veech group of the Chamanara surface, 
on the defining vectors, for which we obtain precise formulas
in \cref{ComputeTheGenerators}. Finally, we deduce
\cref{thm:criterion_finite_index_p=2} from the following more general result
about finite abelian covers of arbitrary degree.

\begin{thm}[Characterization of finite index]\label{thm:characterization_finite_index}
	Let $G$ be a finite abelian group and $h\in \TheSequences_G$ a defining
	vector of a cover $Y_h$. Then $\Gamma(Y_h)$ has finite index in $\Gamma(X)$
	if and only if $h$ is a fixed point of $H^n$ for some $n\in \mathbb{N}$.
\end{thm}

As a consequence of
\cref{thm:characterization_finite_index}, we obtain the following conclusion about devious covers in
the sense of~\cite{hooper_trevino_19}, see \cref{sec:devious} for the relevant
definitions and the proof.

\begin{corintro}[Devious covers have infinite-index Veech groups]\hypertarget{cor:devious}
	{}Let $Y$ be a finite abelian cover of the Chamanara surface $X$. If $\Gamma(Y)$ has finite index in $\Gamma(X)$ then the Teichmüller orbit of $Y$ is periodic. In particular, we have:
	\begin{enumerate}
		\item The Chamanara surface itself has a periodic Teichmüller orbit.
		\item \emph{Devious} finite abelian covers of the Chamanara surface have infinite-index Veech groups.
	\end{enumerate}
\end{corintro}

Examples of devious finite covers were obtained
in~\cite[Section~7]{hooper_trevino_19}, using deck groups inside~$S_n$, the
permutation group of $\{1, \dotsc, n\}$. The above corollary
could be used to look for examples of \emph{abelian} devious finite covers.

Finally, we study which topology the constructed finite covers of the Chamanara
surface have. Whereas all degree--$2$ covers of the Chamanara
surface are Loch Ness monsters or  Jacob's ladders, see \cref{thm:topology_d=2}, we can obtain an arbitrary number of ends when we allow any finite~degree.

\begin{thm}[Topology of covers] \label{thmnew:topology}
	For every $d\in \mathbb{N}$, there exists a finite abelian cover of the
	Chamanara surface with $d$ ends.
\end{thm}

A more detailed version of the statement of \cref{thmnew:topology} can be
  found in \cref{thm:topology}.
  In particular, there and in \cref{prop:criterion_identification_corners}, we relate the number of ends to the deck group $G$ and the defining vector $h$ of the cover.

\paragraph{Outline}
We introduce the Chamanara surface in \cref{sec:chamanara} and discuss Veech
group elements that lift to all covers in \cref{sec:lifting}. In
\cref{sec:monodromy}, we explain how to describe covers with monodromy vectors
and in \cref{sec:action}, we compute the action of $\Gamma(X)$ on these vectors.
Our main theorem (\cref{thm:characterization_finite_index}) is proven in
\cref{sec:characterization_finite_index} and the special case of $d=2$ is
derived in \cref{sec:d=2}. We conclude with the proof of the \hyperlink{cor:devious}{Corollary to Theorem 3} in
\cref{sec:devious} and a discussion of the topology of the covers in \cref{sec:topology}.

\paragraph{Auxiliary files}
Some inspiration for the definition of the set $C_n$ in \cref{sec:characterization_finite_index} came from examples that were produced with the program \cite{program}, written in GAP \cite{GAP}.

\paragraph{Acknowledgements}
We thank the referee for their thorough comments and in particular for their suggestion to use a different generating set in \cref{thm:invariant_set} to shorten the proof.
This research project was supported by AIM’s research program Structured Quartet
Research Ensembles (SQuaREs).
The work of A.~Randecker was partially funded by the Deutsche Forschungsgemeinschaft (DFG, German Research Foundation) -- 441856315.
F.~Valdez thanks grants IN101422 and IN106925 from DGAPA's program PAPIIT for financial support.

\section{The Chamanara surface and its Veech group}\label{sec:chamanara} A
translation surface can be described by considering polygons in the Euclidean
plane whose edges are identified in pairs that differ by translation. The
induced metric on the surface is locally Euclidean, except at the vertices,
where singularities may occur. Classically, this is done with finitely many
polygons but it can also be done with infinitely many polygons or polygons with
infinitely many sides as in our case. The natural action of $\SL(2,\RR)$ on $\RR^2$
induces an action on the set of Euclidean polygons and therefore an action on
the class of translation surfaces.
For more background on finite and infinite translation surfaces, we refer the reader to,~e.g.,~\cite{AM,DHV:infinite}.

In the introduction, we gave a definition of the Veech group as the stabilizer
under the $\SL(2,\RR)$--action. Equivalently, one can define the Veech group as
follows. Consider the group $\Aff(X)$ of affine diffeomorphisms, i.e.,
diffeomorphisms of $X$ which are affine on charts. Then,~$\Gamma(X)$ is the image of
the derivative map $D\colon \Aff(X) \to \GL(2,\R)$. The elements in the kernel
of $D$ are \emph{translations}, i.e., they are diffeomorphisms which are
translations on charts, and we denote $\Trans(X) \coloneq \ker(D)$. In particular, we
have the following short exact sequence
\begin{equation}\label{eq:short_exact_seq}
	1 \to \Trans(X) \to \Aff(X) \xrightarrow{D} \Gamma(X) \to 1.
\end{equation}

The Chamanara surface $X$ is the translation surface that is constructed from a unit
square with gluings as indicated in \cref{fig:chamanara}, see \cite{chamanara_04, herrlich_randecker_16}.
We specify coordinates on the Chamanara surface by defining that the segments $a_n$ are vertical and the segments $b_n$ are horizontal, the origin~$(0,0)$ is in the lower left corner, and the length of the segment $a_n$ as well as of the segment~$b_n$ is $\sfrac{1}{2^n}$ for~$n\in \mathbb{N}$.

\begin{figure}
	\centering
	\begin{tikzpicture}[x=1cm,y=1cm,scale=0.5]
		\newcommand\chamanaraseite{
			\draw (0,0) -- (7.75,0);
			\draw[densely dotted] (7.75,0) -- (8,0);
			\fill (0,0) circle (2pt);
			\fill (4,0) circle (2pt);
			\fill (6,0) circle (2pt);
			\fill (7,0) circle (2pt);
			\fill (7.5,0) circle (2pt);
			\fill (7.75,0) circle (2pt);
		}

		\draw[pattern color=gray!60, pattern=dots] (0,0) -- (8,8) -- (4,8) -- (0,4) -- (0,0);
		\draw[pattern color=gray!60, pattern=dots] (0,0) -- (4,0) -- (8,4) -- (8,8) -- (0,0);

		\draw[pattern color=gray!60, pattern=bricks] (0,4) -- (4,8) -- (2,8) -- (0,6) -- (0,4);
		\draw[pattern color=gray!60, pattern=bricks] (4,0) -- (6,0) -- (8,2) -- (8,4) -- (4,0);

		\draw[pattern color=gray!60, pattern=north west lines] (0,6) -- (2,8) -- (1,8) -- (0,7) -- (0,6);
		\draw[pattern color=gray!60, pattern=north west lines] (6,0) -- (7,0) -- (8,1) -- (8,2) -- (6,0);

		\draw[pattern color=gray!60, pattern=grid] (0,7) -- (1,8) -- (0.5,8) -- (0,7.5) -- (0,7);
		\draw[pattern color=gray!60, pattern=grid] (7,0) -- (7.5,0) -- (8,0.5) -- (8,1) -- (7,0);

		\chamanaraseite
		\path (0,0) -- node[below] {$b_1$} (4,0) -- node[below] {$b_2$} (6,0) -- node[below] {$b_3$} (7,0) -- node[below] {$b_4$} (7.5,0);

		\begin{scope}[xscale=-1,yscale=-1,xshift=-8cm,yshift=-8cm]
			\chamanaraseite
			\path (0,0) -- node[above] {$b_1$} (4,0) -- node[above] {$b_2$} (6,0) -- node[above] {$b_3$} (7,0) -- node[above] {$b_4$} (7.5,0);
		\end{scope}

		\begin{scope}[xscale=-1,rotate=90]
			\chamanaraseite
			\path (0,0) -- node[left] {$a_1$} (4,0) -- node[left] {$a_2$} (6,0) -- node[left] {$a_3$} (7,0) -- node[left] {$a_4$} (7.5,0);
		\end{scope}

		\begin{scope}[xscale=-1,rotate=-90,xshift=-8cm,yshift=-8cm]
			\chamanaraseite
			\path (0,0) -- node[right] {$a_1$} (4,0) -- node[right] {$a_2$} (6,0) -- node[right] {$a_3$} (7,0) -- node[right] {$a_4$} (7.5,0);
		\end{scope}

		\begin{scope}[xshift=12cm]
			\draw[pattern color=gray!60, pattern=dots] (0,0) -- (4,0) -- (8,8) -- (4,8) -- (0,0);

			\draw[pattern color=gray!60, pattern=bricks] (0,0) -- (4,8) -- (2,8) -- (0,4) -- (0,0);
			\draw[pattern color=gray!60, pattern=bricks] (4,0) -- (6,0) -- (8,4) -- (8,8) -- (4,0);

			\draw[pattern color=gray!60, pattern=north west lines] (0,4) -- (2,8) -- (1,8) -- (0,6) -- (0,4);
			\draw[pattern color=gray!60, pattern=north west lines] (6,0) -- (7,0) -- (8,2) -- (8,4) -- (6,0);

			\draw[pattern color=gray!60, pattern=grid] (0,6) -- (1,8) -- (0.5,8) -- (0,7) -- (0,6);
			\draw[pattern color=gray!60, pattern=grid] (7,0) -- (7.5,0) -- (8,1) -- (8,2) -- (7,0);

			\draw[pattern color=gray!60, pattern=crosshatch dots] (0,7) -- (0.5,8) -- (0.25,8) -- (0,7.5) -- (0,7);
			\draw[pattern color=gray!60, pattern=crosshatch dots] (7.5,0) -- (7.75,0) -- (8,0.5) -- (8,1) -- (7.5,0);

			\chamanaraseite
			\path (0,0) -- node[below] {$b_1$} (4,0) -- node[below] {$b_2$} (6,0) -- node[below] {$b_3$} (7,0) -- node[below] {$b_4$} (7.5,0);

			\begin{scope}[xscale=-1,yscale=-1,xshift=-8cm,yshift=-8cm]
				\chamanaraseite
				\path (0,0) -- node[above] {$b_1$} (4,0) -- node[above] {$b_2$} (6,0) -- node[above] {$b_3$} (7,0) -- node[above] {$b_4$} (7.5,0);
			\end{scope}

			\begin{scope}[xscale=-1,rotate=90]
				\chamanaraseite
				\path (0,0) -- node[left] {$a_1$} (4,0) -- node[left] {$a_2$} (6,0) -- node[left] {$a_3$} (7,0) -- node[left] {$a_4$} (7.5,0);
			\end{scope}

			\begin{scope}[xscale=-1,rotate=-90,xshift=-8cm,yshift=-8cm]
				\chamanaraseite
				\path (0,0) -- node[right] {$a_1$} (4,0) -- node[right] {$a_2$} (6,0) -- node[right] {$a_3$} (7,0) -- node[right] {$a_4$} (7.5,0);
			\end{scope}
		\end{scope}
	\end{tikzpicture}
	\caption{Chamanara surface with a cylinder decomposition of slope $1$ (left) and of slope $2$ (right).}
	\label{fig:cylinder_decomposition_chamanara_slope}
\end{figure}

The Chamanara surface has one singularity (which induces one end of the topological surface) and countably many directions in
which the surface can be decomposed into cylinders. We recall that a cylinder is
a subset which is isometric to $\mathbb{R} / w\ZZ \times [0,h]$, where $w$ is
called the width and $h$ the height. The modulus of a cylinder is the ratio
$\sfrac{h}{w}$.
In \cref{fig:cylinder_decomposition_chamanara_slope}, two cylinder decompositions are displayed. They are not images of each other by Veech group elements: In the first type of
decomposition, all cylinders have the same modulus. In the second type of
decomposition, the modulus of the middle cylinder is three times larger than all
the other~moduli. In fact, there exist cylinder decompositions for every slope $2^n$ with $n\in \mathbb{Z}$. Furthermore, the cylinder decompositions of slope $2^n$ with $n$ even form an orbit under the action of the subgroup generated by the element $H$ described below. Analogously, the cylinder decompositions of slope~$2^n$ with $n$ odd form an orbit of $\left\langle H \right\rangle$.

From the cylinder decompositions of slope $1$ and of slope $2$, we obtain the
parabolic elements
\begin{equation*}
	P_1 =
	\begin{pmatrix}
		4 & -3 \\
		3 & -2
	\end{pmatrix}
	\qquad
	\text{and}
	\qquad
	P_2 =
	\begin{pmatrix}
		4 & -\frac{3}{2} \\
		6 & -2
	\end{pmatrix}
\end{equation*}
in the Veech group. The corresponding affine homeomorphisms fix the boundaries of the respective cylinders pointwise and act as Dehn multitwists with respect to the core curves.
Chamanara showed in~\cite[Theorem 4]{chamanara_04}
that the Veech group~$\Gamma(X)$ is generated by $P_1, P_2$, and $-\Id$, see also \cite{herrlich_randecker_16}. In particular, the \emph{projective Veech group} $P\Gamma(X) \coloneqq
\Gamma(X) / \{\pm \Id\}$ is a free group, freely generated by~$P_1$ and $P_2$.

From the description in \cref{fig:chamanara}, it can be deduced geometrically that
\begin{equation*}
	H = \begin{pmatrix}2&0\\0&\frac{1}{2}\end{pmatrix}
\end{equation*}
is an element of the Veech group $\Gamma(X)$. A direct computation shows that $H = -P_1P_2$.

A \emph{saddle connection} of a translation surface $X$ is a geodesic segment whose interior lies in~$X$ and both boundary
points are singularities.
Since $b_1$ as in \cref{fig:chamanara} is the unique horizontal saddle
connection of length $\sfrac{1}{2}$, the translation group of the Chamanara
surface is trivial. Hence, from the short exact
sequence~\eqref{eq:short_exact_seq}, we have that $\Aff(X) \cong \Gamma(X)$.

\begin{figure}
	\centering
	\begin{tikzpicture}[x=1cm,y=1cm, scale=0.5,
		   mid arrow/.style={
			   postaction={decorate,decoration={
				   markings,
				   mark=at position .5 with {\arrow{latex}}
		   }}
		 }]
	 \draw (0,0) -- node[below] {$b_1$}
		   (4,0) -- node[below] {$b_2$}
		   (6,0) -- node[below] {$b_3$}
		   (7,0) -- node[below] {$b_4$}
		   (7.5,0) -- (7.75,0);
	 \draw[densely dotted] (7.75,0) -- (8,0);

	 \fill (0,0) circle (2pt);
	 \fill (8,8) circle (2pt);

	 \fill (4,0) circle (2pt);
	 \fill (6,0) circle (2pt);
	 \fill (7,0) circle (2pt);
	 \fill (7.5,0) circle (2pt);
	 \fill (7.75,0) circle (2pt);

	 \draw (8,8) -- node[above] {$b_1$}
		   (4,8) -- node[above] {$b_2$}
		   (2,8) -- node[above] {$b_3$}
		   (1,8) -- node[above] {$b_4$}
		   (0.5,8) -- (0.25,8);
	 \draw[densely dotted] (0.25,8) -- (0,8);

	 \fill (4,8) circle (2pt);
	 \fill (2,8) circle (2pt);
	 \fill (1,8) circle (2pt);
	 \fill (0.5,8) circle (2pt);
	 \fill (0.25,8) circle (2pt);

	 \draw (0,0) -- node[left] {$a_1$}
		   (0,4) -- node[left] {$a_2$}
		   (0,6) -- node[left] {$a_3$}
		   (0,7) -- node[left] {$a_4$}
		   (0,7.5) -- (0,7.75);
	 \draw[densely dotted] (0,7.75) -- (0,8);

	 \fill (0,4) circle (2pt);
	 \fill (0,6) circle (2pt);
	 \fill (0,7) circle (2pt);
	 \fill (0,7.5) circle (2pt);
	 \fill (0,7.75) circle (2pt);

	 \draw (8,8) -- node[right] {$a_1$}
		   (8,4) -- node[right] {$a_2$}
		   (8,2) -- node[right] {$a_3$}
		   (8,1) -- node[right] {$a_4$}
		   (8,0.5) -- (8,0.25);
	 \draw[densely dotted] (8,0.25) -- (8,0);

	 \fill (8,4) circle (2pt);
	 \fill (8,2) circle (2pt);
	 \fill (8,1) circle (2pt);
	 \fill (8,0.5) circle (2pt);
	 \fill (8,0.25) circle (2pt);

	 \fill (4,4) node[below] {$x$} circle (3pt);
	 % alpha_1
	 \draw[densely dashed, mid arrow] (4,4) .. controls +(35:1cm) and +(-150:0.7cm) .. (6,5.2) node[below] {$\alpha_1$}
	 .. controls +(30:0.5cm) and +(180:1cm) .. (8,5.5);
	 \draw[densely dashed] (0,1.5) .. controls +(0:1cm) and +(-130:0.7cm) .. (3,3)
	 .. controls +(50:0.5cm) and +(-145:1cm) .. (4,4);
	 % alpha_2
	 \draw[loosely dashed, mid arrow] (4,4) .. controls +(-30:1cm) and +(160:0.7cm) .. (6,3.2) node[below] {$\alpha_2$}
	 .. controls +(-20:0.5cm) and +(180:1cm) .. (8,3);
	 \draw[loosely dashed] (0,5) .. controls +(0:1cm) and +(160:0.7cm) .. (3,4.6)
	 .. controls +(-20:0.7cm) and +(150:1cm) .. (4,4);
	 % beta_1
	 \draw[densely dotted] (5.5,8) .. controls +(-90:0.5cm) and +(70:0.5cm) .. (5.3,6.4) node[right] {$\beta_1$}
	 .. controls +(-110:1cm) and +(50:1cm) .. (4,4);
	 \draw[densely dotted, postaction={decorate,decoration={
	 		markings,
	 		mark=at position .5 with {\arrow{latex reversed}}}}] (1.5,0) .. controls +(90:0.2cm) and +(-150:1.4cm) .. (2.8,1.6)
	 .. controls +(30:0.9cm) and +(-130:0.4cm) .. (4,4);
	 % beta_2
	 \draw[dotted] (3,8) .. controls +(-90:0.5cm) and +(100:0.3cm) .. (3.1,6.4) node[right] {$\beta_2$}
	 .. controls +(-80:1cm) and +(120:1cm) .. (4,4);
	 \draw[dotted, postaction={decorate,decoration={
	 		markings,
	 		mark=at position .5 with {\arrow{latex reversed}}}}] (5,0) .. controls +(90:0.5cm) and +(-80:1cm) .. (4.8,1.8)
	 .. controls +(100:0.3cm) and +(-60:0.7cm) .. (4,4);
	\end{tikzpicture}
	\caption{The first generators of $\pi_1(X)$.}
	\label{fig:chamanara_generators}
   \end{figure}

For later use, we now fix the notation for the curves that define the fundamental
group of the Chamanara surface.
% \begin{definition}\label{TheIsomorphism}
Let $x$ be the mid point of the square which defines the Chamanara surface. For
$i \in \NN = \{1, 2, \ldots\}$, let~$\alpha_i$ and $\beta_i$, respectively, be the oriented
closed curves with base point $x$ crossing~$a_i$ from left to right,
respectively~$b_i$ from top to bottom, and having no other intersection points
with the boundary of the square (see \cref{fig:chamanara_generators}).

Observe that $\pi_1(X)$ is freely generated by the closed curves $\alpha_i$ and
$\beta_i$ with $i \in \NN$ (cf.~\cref{fig:chamanara_generators}). Since the open square is simply connected, we obtain a decomposition of an element $\gamma \in \pi_1(X)$ as word in these generators and their inverses as follows: If $\gamma$ crosses the edge $a_j$ from left to right then note down $\alpha_j$. It $\gamma$ crosses the edge $a_j$ from right to left then note down $\alpha_j^{-1}$. If it crosses $b_j$ from top to bottom, note down $\beta_j$, and if it crosses $b_j$ from bottom to top, note down~$\beta_j^{-1}$. Then, the product of this sequence of $\alpha_j$'s,  ${\alpha_j}^{-1}$'s, $\beta_j$'s and ${\beta_j}^{-1}$'s is $\gamma$.

Let $I = \ZZ\backslash\{0\}$ and  let $F(I)$ be
the free group  over $I$. We fix the isomorphism $\rho\colon F(I) \to
\pi_1(X)$~defined by:
	\begin{equation}\label{isomorphism_rho}
		\begin{split}
			\rho \colon F(I) &\to \pi_1(X) \\
			i &\mapsto \begin{cases}
				\alpha_i, & \text{if } i > 0, \\
				\beta_{-i}, & \text{if } i < 0.
			\end{cases}
		\end{split}
	\end{equation}

Here, and in the whole article, with a slight abuse of notation, we denote the homotopy class of the curve \(\gamma\) simply by $\gamma$.

\section{Descending and lifting of Veech group elements} \label{sec:lifting}

We now consider coverings  $p \colon Y \to X$ of the Chamanara surface $X$. Throughout
the article, covers will always be connected. Recall that the Chamanara surface $X$ by its construction comes with a translation atlas \(\mathcal{U}\). Furthermore, for any topological covering $p \colon Y \to X$, there is a unique translation atlas $\mathcal{V}$ on $Y$, namely $\mathcal{V} = p^*\mathcal{U}$,  such that $p \colon (Y,\mathcal{V}) \to (X,\mathcal{U})$ is a \emph{translation covering}, i.e., locally on charts, it is a translation.

\begin{rem}[Translation structure on covers]
	\label{RemLiftStructure}
	When we consider a covering $p \colon Y \to X$ of the Chamanara surface $X$, we will always consider $Y$ with the translation structure  $\mathcal{V} = p^*\mathcal{U}$.
\end{rem}

In general, given a covering $Y \to X$ of translation surfaces, it is neither true
that $\Gamma(Y) \subseteq \Gamma(X)$ nor that $\Gamma(X) \subseteq \Gamma(Y)$.
In the case that $X$ is the Chamanara surface, however, we do obtain the first relation,
see \cref{prop:everything_descends}.
This makes it possible to describe \(\Gamma(Y)\) as  in \cref{prop:VeechGroupOfCovering}. Here we use the fact that the group \(\Trans(X)\) of translations of \(X\) is trivial
(see~\cref{sec:chamanara}) and thus the derivative map $D\colon \Aff(X) \to \Gamma(X)
\subseteq \SL(2,\R)$ is an isomorphism. We denote the preimage of a
matrix \(A\in\Gamma(X)\) in \(\Aff(X)\)  by \(f_A\).

The proof of  \cref{prop:everything_descends} relies on the crucial property of $X$ that all translations of its universal cover $\widetilde{X}$ are deck transformations as detailed in the next lemma.

\begin{lem}[Translations of the universal cover are deck transformations]
	\label{lem:aut=deck}
	Let $\widetilde{X} \to X$ be the universal cover of the Chamanara surface
	$X$ endowed with the translation surface structure obtained by pull-back via the covering map. Then $\Trans(\widetilde{X})$ equals $\Deck(\widetilde{X}/X)$, where
	$\Deck(\widetilde{X}/X)$ is the group of deck transformations of the
	universal cover.
\end{lem}

\begin{proof}
	As any deck transformation is locally a translation, we have that
	$\Deck(\widetilde{X}/X) \leq \Trans(\widetilde{X})$. To show
	$\Trans(\widetilde{X}) \leq \Deck(\widetilde{X}/X)$, note that the segment
	$b_1$ (cf.~\cref{fig:chamanara})  is the unique horizontal saddle connection
	of length $\sfrac{1}{2}$ on the Chamanara surface. Thus, any translation of
	$\widetilde{X}$ must permute the lifts of~$b_1$. Consider $t \in
	\Trans(\widetilde{X})$ and a lift $b$ of~$b_1$. There is a unique deck
	transformation $d$ which takes~$t(b)$ to~$b$. Since each deck transformation
	is a translation, the derivative of $d \circ t$ must be the identity element~$\Id$ of $\SL_2 (\mathbb{R})$. Moreover, any point on the saddle
	connection $b$ is a fixed point of the affine transformation $d \circ t$.
	The only such transformation is the identity, and so $t=d^{-1}$, and $t$ is
	a deck transformation.
\end{proof}

With \cref{lem:aut=deck} in hand, we prove that every affine
diffeomorphism of a cover of~\(X\) descends to \(X\).

\begin{prop}[Everything descends property]
	\label{prop:everything_descends}
	The Chamanara surface $X$ has the property that for every covering $Y \to
	X$, every element in the affine group of $Y$ descends. In particular,
	$\Gamma(Y) \subseteq \Gamma(X)$.
\end{prop}

\begin{proof}
	Let $\widetilde{X}$ be the universal cover of the Chamanara surface. First
	note that every diffeo\-morphism of $Y$ lifts to the universal cover.
        Therefore, it suffices to show that all affine
	diffeomorphisms of $\widetilde{X}$ descend to affine diffeomorphisms of $X$.
	Hence, we have to show that for every affine diffeomorphism $f \in
	\Aff(\widetilde{X})$ and for every deck transformation $d \in
	\Deck(\widetilde{X}/X)$, there exists a deck transformation~$d' \in
	\Deck(\widetilde{X}/X)$ such that $d' \circ f = f \circ d$.
	Consider $d' = f \circ d \circ f^{-1}$. The
	derivative of this map equals $D(f)\cdot \Id \cdot D(f)^{-1} = \Id$, where
	$\Id \in \SL(2,\R)$ is the identity matrix. Hence $d' = f \circ d \circ f^{-1}$
	is a translation, i.e., lies in $\Trans(\widetilde{X})$. By
	\cref{lem:aut=deck}, we have $\Trans(\widetilde{X}) =
	\Deck(\widetilde{X}/X)$ which shows the claim.
\end{proof}

\begin{cor}[All Veech group elements of the cover are lifts] \label{prop:VeechGroupOfCovering}
	Let \(p \colon Y \to X\) be a covering of the Chamanara surface \(X\). Then
	\begin{equation*}
		\Gamma(Y) = \{A \in \Gamma(X) \mid f_A \text{ lifts to } Y \text{ via } p\}.
	\end{equation*}
\end{cor}

\begin{proof}
	This follows directly from \cref{prop:everything_descends}.
\end{proof}

To conclude this section, we describe some elements in $\Gamma(X)$ that lift to
any finite normal cover of \(X\) of fixed degree $d$.

\begin{prop}[Powers of parabolic elements lift]
	\label{prop:powers_parabolic_lift} Let $p \colon Y \to X$ be a finite normal
	covering of degree $d$. Then $P_1^d$ and $P_2^d$ are contained in
	$\Gamma(Y)$.
\end{prop}

\begin{proof}
	Recall from \cref{sec:chamanara} that $P_1$ and $P_2$ act as Dehn multitwists on $X$ about the core curves of the cylinders of slope $1$
	and of slope $2$, respectively. Let $C$ be a cylinder of slope $1$
	or $2$ in~$X$. The preimage of $C$ is a set of cylinders with the same
	height as $C$ and the circumference an integer multiple of the
	circumference of $C$. As the covering is normal, this multiple is a
	divisor of $d$. Hence, $P_1^d$ or $P_2^d$ acts as (possibly a power of) a Dehn twist on any cylinder in the preimage of~$C$ which fixes the boundary pointwise. This gives a well-defined diffeomorphism of the whole surface~$X$.
\end{proof}

For the third generator, $-\Id$, we will see in  \cref{rem:Id_lifts}\ that $-\Id$ lifts for every abelian cover.
Therefore, every generator has a power that lifts to $Y$.

\section{Describing coverings through monodromy}\label{sec:monodromy}

In this section, we introduce our main character: the space $\Cov_G$ of
isomorphism classes of normal coverings $\p\colon Y \to X$ of the Chamanara surface
$X$ with deck group $G$. Later in this section,
 \(G\) will be a finite abelian group. We begin by giving a general description of~$\Cov_G$ in \cref{DescriptionBG} which results from classical covering theory, similar to the one given in~\cite[Section 4.1]{hooper_trevino_19}.
 Subsequently, in  \cref{Determine_h},  we explain how to compute for abelian $G$ a vector $h$ that captures all the information about a normal covering \(p\) through the monodromy map of~\(\p\). We refer to~\cite{hatcher_02} as a standard reference for covering theory.

For the definition of \(\Cov_G\), recall that in this article, covers are by definition connected. This is one aspect in which our space $\Cov_G$ differs from the space of covers in \cite{hooper_trevino_19}; the other is that we consider only normal covers.

Recall that two coverings $\p_1\colon Y_1 \to X$ and $\p_2\colon Y_2 \to X$ are
\emph{isomorphic} if there is a homeo\-morphism $h\colon Y_1 \to Y_2$ such that $\p_2
\circ h = \p_1$. We then write~$\p_1 \sim \p_2$ and denote by $[\p]$ the
equivalence class of $\p$. With this notation, $\Cov_G$ is defined as~follows.

\begin{definition}[The space \boldmath{$\Cov_G$}]
	For a group $G$, let
	\[
		\Cov_G = \{\p\colon Y \to X \mid \p \text{ is a normal covering with deck group } G\}/\! \sim.
	\]
\end{definition}

\begin{prop}[Correspondence between normal coverings and monodromy
	vectors]\label{DescriptionBG} Let \(G\) be a group. There is a one-to-one correspondence between
	\(\Cov_G\) and $\TheSequencesQuot_G = \TheSequences_G/\!\sim$, where
	\begin{equation}\label{eq:TheSequence}
	  \TheSequences_G =  \{(h_i) \in G^{\ZZ} \mid \left\langle h_i | i \in \ZZ \right\rangle = G \text{ and }
	h_0 = 0 \}
	\end{equation}
	and \((h_i) \sim (h'_i)\) if there exists an automorphism \(q \in \Aut(G)\)
	such that \(h'_i = q(h_i)\)  for all \(i \in \ZZ\).
\end{prop}

\begin{rem}[Disconnected covers]
	We will see in the proof of \cref{DescriptionBG} that 
	we have a correspondence between the set of connected and disconnected normal covers with deck group~$G$ 
	up to isomorphism and $G^{\ZZ}/\!\sim$. The condition that $\left\langle h_i | i \in \ZZ \right\rangle = G$ is equivalent to the fact that for a corresponding covering $p:Y \to X$, the surface~$Y$ is connected.
	In particular, the vector~$h = \underline{0}$ is always excluded from $\TheSequences_G$.
\end{rem}

	\begin{proof}[Proof of \cref{DescriptionBG}]
		The desired one-to-one correspondence \(\Theta\colon \Cov_G \to \TheSequencesQuot_G \) naturally arises as a composition \(\Theta = \Theta_3 \circ
		\Theta_2 \circ \Theta_1\) of bijections \(\Theta_1,\Theta_2, \Theta_3\) which we describe in the following.
		
		To define \(\Theta_1\), recall that any covering $\p\colon Y \to X$ induces an
embedding $\p_\star \colon \pi_1(Y) \into \pi_1(X)$.
The embedding depends only up to conjugation on the chosen base points of the fundamental groups and for isomorphic coverings, we obtain conjugated subgroups of \(\pi_1(X)\).
Furthermore, the image of $\pi_1(Y)$ in~$\pi_1(X)$
is normal if and only if the covering $\p$ is normal.  Hence, for normal coverings, the
image of the embedding of $\pi_1(Y)$ in $\pi_1(X)$ is independent of the choice
of base points and of the representative in the isomorphism class.
Finally, we have that $\Deck(\p) \cong \pi_1(X)/\pi_1(Y)$ for the deck group $\Deck(\p)$
of the covering $\p \colon Y \to X$.
This gives the~map
\[
	\Theta_1\colon \Cov_G \;\to\; \NNN_G = \{N \trianglelefteq \pi_1(X) \mid \pi_1(X)/N \; \cong \; G \},\quad
	  [\p\colon Y \to X] \; \mapsto\; \p_\star(\pi_1(Y)),
\]
which is a one-to-one correspondence by standard covering theory.

To define \(\Theta_2\), observe that any normal subgroup $N \in
\NNN_G$ defines a projection $\kappa_N : \pi_1(X) \to G$ such that its kernel is
$N$. For any other projection $\kappa'_N\colon \pi_1(X) \to G$ with the same kernel~$N$, there is an automorphism $\alpha \in \Aut(G)$ with $\alpha \circ \kappa_N =
\kappa'_N$. Hence we have a one-to-one correspondence
\[
	\Theta_2\colon \NNN_G \to \HHH_G = \{\kappa:\pi_1(X) \to G \mid \kappa \text{ is surjective}\}/\!\sim, \quad
	 N \mapsto [\kappa_N].
\]
Here, $\sim$ is the equivalence relation $\kappa \sim \kappa' \iff
\exists \alpha \in \Aut(G) \mbox{ with } \kappa' = \alpha \circ \kappa$ and
$[\kappa]$ is the equivalence class of $\kappa$.

Finally, to define \(\Theta_3\), we assign to each element in
$\HHH_G$ a bi-infinite sequence as follows. Recall from \eqref{isomorphism_rho}
the bijection \(\rho\) between the fundamental group~$\pi_1(X)$
with chosen generators~\(\alpha_i\) and~\(\beta_i\)~(\(i \in
\NN\)) and the free group $F(I)$ in the infinite countable set $I =
\ZZ\backslash\{0\}$. For any surjective homomorphism $\kappa\colon \pi_1(X) \to G$, we define  $h$ in $G^{\ZZ}$ as follows:
  \begin{equation}\label{Sequence}
    h_i = \begin{cases}
		\kappa(\alpha_i), &\text{ if } i > 0,\\
		0 &\text{ if } i = 0,\\
		\kappa(\beta_{-i}) &\text{ if } i < 0.
	\end{cases}
  \end{equation}
  Here, $0 = 0_G$ is the identity element in the group $G$. Observe that \(h\) lies in \(\TheSequences_G\), as defined in \cref{eq:TheSequence},
  and we obtain a well-defined map
\[
\Theta_3\colon \HHH_G \to\TheSequencesQuot_G =  \TheSequences_G/\!\sim, \quad %\\
      [\kappa]  \mapsto [h],
\]
where \(\sim\) is the equivalence relation defined in the statement of \cref{DescriptionBG} and
\([\cdot]\) denotes the equivalence class. By definition of \(\HHH_G\) and
\(\TheSequencesQuot_G\), \(\Theta_3\) is a bijection.

Composing all three bijections from above, we obtain the one-to-one correspondence $ \Theta
\coloneqq \Theta_3\circ\Theta_2\circ\Theta_1$:
\begin{equation*}
	\Theta\colon \Cov_G \xrightarrow{\Theta_1} \NNN_G \xrightarrow{\Theta_2} \HHH_G \xrightarrow{\Theta_3}  \TheSequencesQuot_G \;
	, \quad
	[\p] \mapsto [h] = \Theta([\p])
	. \qedhere
\end{equation*}
\end{proof}

\begin{definition}[Defining vector and defining group homomorphism]\label{DefiningVector}
If \(h \in \TheSequences_G\) is
a representative  of \(\Theta([p])\) then we call \(h\) a \emph{defining vector of~\(p\)}.
Moreover, we call any \(\kappa\colon \pi_1(X) \to G\) with
\(\Theta_2\circ\Theta_1([\p]) = [\kappa]\) a \emph{defining group homomorphism
for $\p$}.
\end{definition}

We now describe how a representative \(h\) of \(\Theta([p])\) can be computed from the monodromy map of
\(p\) in case the covering~\(p\) is abelian, i.e., it is normal and the deck group $G$ is abelian.
Let us first briefly recall the notion of a \emph{monodromy map}. Let $x$ be a
point in $X$ and let  $F_x$ be its fibre in $Y$.  For any closed path~$\gamma$ on~$X$ with
starting point $x$ and for any~$y \in F_x$, denote by $y \cdot \gamma$ the end point
of the unique lift of~$\gamma$ with starting point $y$. Then, the \emph{monodromy
map} is the map $m_{\p}\colon \pi_1(X,x) \to \Sym(F_x)$ which assigns to each $\gamma \in
\pi_1(X,x)$ the permutation $m_{\p}(\gamma)\colon y \mapsto y \cdot \gamma$.  The map $m_{\p}$
is an anti-homomorphism since~$\pi_1(X,x)$ acts via~$m_{\p}$ from the right.
Since the covering $\p$ is normal
with deck group~$G = \Deck(\p)$, for any $y,
y' \in F_x$, there exists a unique~$g \in G$ such that~$g\cdot y = y'$. If we fix~\(y_0 \in F_x\), we obtain a group homomorphism~\(\kappa_p\)~with
\begin{equation}\label{kappap}
	\begin{split}
		\kappa_p\colon \pi_1(X,x) &\to G\\
		 \gamma &\mapsto g,
	\end{split}
\end{equation}
where $g$ is the unique element in $G$ with $g\cdot y_0 = y_0\cdot \gamma$. By
definition, the kernel of $\kappa_p$ is $\pi_1(Y)$, therefore~\(\kappa_p\) is a
defining group homomorphism for the covering $\p$, i.e., we have
$\Theta_2\circ\Theta_1([\p]) = [\kappa_{\p}]$. Moreover, by definition of
$\kappa_{\p}$, we have the identity
\begin{equation}\label{monodromy}
	\kappa_{\p}(\gamma)\cdot y_0 = y_0\cdot \gamma = m_{\p}(\gamma)(y_0).
\end{equation}
For non-abelian groups, this identity only holds for $y_0$. However, since we assume that $G$ is abelian, we may use the standard fact of covering theory phrased in the following lemma to extend it to the whole fibre $F_x$.

\begin{lem}[Monodromy action and action by deck elements coincide for abelian coverings]
  \label{MonodromyAbelian}
	Let $\p \colon Y \to X$ be an abelian covering with deck group~$G$,
	\(x \in X\), $\kappa_{\p}$ be the group homomorphism defined
	by~\eqref{kappap}, and~$m_{\p}$ the monodromy map. Then, for every
	$\gamma \in\pi_1(X,x)$ and every $y\in F_x$, we have that
	\[
		\kappa_{\p}(\gamma)\cdot y = y\cdot \gamma = m_{\p}(\gamma)(y).
	\]

\begin{proof}
	Let $g = \kappa_{\p}(\gamma)$, hence $g \cdot y_0 = y_0\cdot \gamma$. Since $\p$ is
	normal, for any $y \in F_x$, there is some~$g' \in G$ such that $y = g' \cdot
	y_0$. Recall that  $y_0\cdot \gamma$ is the end point of the lift $\hat{\gamma}$ of $\gamma$ which
	starts at $y_0$. Then~$g'\cdot \hat{\gamma}$ is the lift of $\gamma$ starting  at $y = g' \cdot
	y_0$  and ending in $g' \cdot (y_0\cdot \gamma)$. Altogether, we have
	\begin{equation*}
		m_{\p}(\gamma)(y) = y \cdot \gamma = g' \cdot y_0 \cdot \gamma = g' \cdot g \cdot y_0 = g \cdot g' \cdot y_0 = g \cdot y = \kappa_{\p}(\gamma) \cdot y
		. \qedhere
	\end{equation*}
\end{proof}
\end{lem}

In the next lemma, we describe how $[h]$ can be determined from $p$.

\begin{lem}[Determining the defining vector for an abelian covering]\label{Determine_h}
  Let $\p \colon Y \to X$ be an abelian covering with deck group $G$, let $x$ be the
  mid point of the square forming the Chamanara surface as in
  \cref{fig:chamanara}, and~$y$ an arbitrary preimage of $x$ on $Y$.  Define $h
  \in  G^{\ZZ}$ as:
		\[
			h_i\cdot y =
			\begin{cases}
				y\cdot \alpha_i, & \text{if } i > 0,\\
				0, & \text{if } i = 0,\\
				y\cdot \beta_{-i}, & \text{if } i < 0.
			\end{cases}
		\]
		Then, $\Theta([\p]) = [h]$.
\begin{proof}
	We choose $\kappa_{\p}$ as defined in \cref{kappap}. As described after \cref{kappap}, we then have $\Theta_2
	\circ \Theta_1([\p]) = [\kappa_{\p}]$. Moreover, by \cref{Sequence} and \cref{MonodromyAbelian}, we
	have that $[h] = \Theta_3([\kappa_{\p}])$. This shows the claim.
\end{proof}
\end{lem}

\begin{figure}[t]
	\centering
	\begin{tikzpicture}[x=1cm,y=1cm, scale=0.5]
		\draw (0,0) -- node[below] {$b_1'$}
		(4,0) -- node[below] {$b_2$}
		(6,0) -- node[below] {$b_3'$}
		(7,0) -- node[below] {$b_4$}
		(7.5,0) -- (7.75,0);
		\draw[densely dotted] (7.75,0) -- (8,0);

		\fill (0,0) circle (2pt);
		\fill (8,8) circle (2pt);
		
		\fill (4,0) circle (2pt);
		\fill (6,0) circle (2pt);
		\fill (7,0) circle (2pt);
		\fill (7.5,0) circle (2pt);
		\fill (7.75,0) circle (2pt);

		\draw (8,8) -- node[above] {$b_1$}
		(4,8) -- node[above] {$b_2$}
		(2,8) -- node[above] {$b_3$}
		(1,8) -- node[above] {$b_4$}
		(0.5,8) -- (0.25,8);
		\draw[densely dotted] (0.25,8) -- (0,8);

		\fill (4,8) circle (2pt);
		\fill (2,8) circle (2pt);
		\fill (1,8) circle (2pt);
		\fill (0.5,8) circle (2pt);
		\fill (0.25,8) circle (2pt);

		\draw (0,0) -- node[left] {$a_1'$}
		(0,4) -- node[left] {$a_2$}
		(0,6) -- node[left] {$a_3'$}
		(0,7) -- node[left] {$a_4$}
		(0,7.5) -- (0,7.75);
		\draw[densely dotted] (0,7.75) -- (0,8);

		\fill (0,4) circle (2pt);
		\fill (0,6) circle (2pt);
		\fill (0,7) circle (2pt);
		\fill (0,7.5) circle (2pt);
		\fill (0,7.75) circle (2pt);

		\draw (8,8) -- node[right] {$a_1$}
		(8,4) -- node[right] {$a_2$}
		(8,2) -- node[right] {$a_3$}
		(8,1) -- node[right] {$a_4$}
		(8,0.5) -- (8,0.25);
		\draw[densely dotted] (8,0.25) -- (8,0);

		\fill (8,4) circle (2pt);
		\fill (8,2) circle (2pt);
		\fill (8,1) circle (2pt);
		\fill (8,0.5) circle (2pt);
		\fill (8,0.25) circle (2pt);
		
		\begin{scope}[xshift = 12cm]
					\draw (0,0) -- node[below] {$b_1$}
			(4,0) -- node[below] {$b_2'$}
			(6,0) -- node[below] {$b_3$}
			(7,0) -- node[below] {$b_4'$}
			(7.5,0) -- (7.75,0);
			\draw[densely dotted] (7.75,0) -- (8,0);

			\fill (0,0) circle (2pt);
			\fill (8,8) circle (2pt);
			
			\fill (4,0) circle (2pt);
			\fill (6,0) circle (2pt);
			\fill (7,0) circle (2pt);
			\fill (7.5,0) circle (2pt);
			\fill (7.75,0) circle (2pt);

			\draw (8,8) -- node[above] {$b_1'$}
			(4,8) -- node[above] {$b_2'$}
			(2,8) -- node[above] {$b_3'$}
			(1,8) -- node[above] {$b_4'$}
			(0.5,8) -- (0.25,8);
			\draw[densely dotted] (0.25,8) -- (0,8);

			\fill (4,8) circle (2pt);
			\fill (2,8) circle (2pt);
			\fill (1,8) circle (2pt);
			\fill (0.5,8) circle (2pt);
			\fill (0.25,8) circle (2pt);

			\draw (0,0) -- node[left] {$a_1$}
			(0,4) -- node[left] {$a_2'$}
			(0,6) -- node[left] {$a_3$}
			(0,7) -- node[left] {$a_4'$}
			(0,7.5) -- (0,7.75);
			\draw[densely dotted] (0,7.75) -- (0,8);

			\fill (0,4) circle (2pt);
			\fill (0,6) circle (2pt);
			\fill (0,7) circle (2pt);
			\fill (0,7.5) circle (2pt);
			\fill (0,7.75) circle (2pt);

			\draw (8,8) -- node[right] {$a_1'$}
			(8,4) -- node[right] {$a_2'$}
			(8,2) -- node[right] {$a_3'$}
			(8,1) -- node[right] {$a_4'$}
			(8,0.5) -- (8,0.25);
			\draw[densely dotted] (8,0.25) -- (8,0);

			\fill (8,4) circle (2pt);
			\fill (8,2) circle (2pt);
			\fill (8,1) circle (2pt);
			\fill (8,0.5) circle (2pt);
			\fill (8,0.25) circle (2pt);
		\end{scope}
	\end{tikzpicture}
	\caption{The cover of the Chamanara surface with deck group $G = \mathbb{Z}/2\mathbb{Z}$ and defining vector $h = (\ldots, 0, 1, 0, 1, h_0 = 0, 1, 0, 1, 0, \ldots)$. As usual, identification of segments is indicated by the labels.}
	\label{fig:alternating_cover}
\end{figure}

The defining vector provides a visual description of the covering surface as it indicates for each segment on the boundary of a square the copy to which this segment is glued. For example, the constant vector $h = (\ldots, 0, 0, h_0 = 0, 0, 0 , \ldots)$ would define the disconnected cover with~$|G|$ connected components which we exclude in our considerations. In contrast, the vector $h = (\ldots, g, g, h_0 = 0, g, g, \ldots)$ for a generator $g\in G$ defines the cover where every boundary segment on the top or right edge of a square labelled with $g' \in G$ is glued to the corresponding boundary segment on the bottom or left edge of the copy $g' + g$. An explicit example for~$G = \mathbb{Z}/2\mathbb{Z}$ and $h = (\ldots, 0, 1, 0, 1, h_0 = 0, 1, 0, 1, 0, \ldots )$ is given in \cref{fig:alternating_cover}.

\begin{rem}[Notation for coverings]
  In the rest of this article, we identify $\Cov_G$ with $\TheSequencesQuot_G$ via~$\Theta$ and denote an element~$[\p]$ of $\Cov_G$ as the class $[h] = \Theta([\p])\in \TheSequencesQuot_G$ of a bi-infinite vector $h$ with coefficients in $G$.
  Furthermore, we denote the covering surface of $p$ by $Y_h$  and consider it always with the unique translation structure such that $p \colon Y_h \to X$ becomes a translation covering (cf.~\cref{RemLiftStructure}). We denote the Veech group of $Y_h$ by $\Gamma(Y_h)$.
\end{rem}

\section{Action of the Veech group on \texorpdfstring{$\Cov_G$}{Cov G}} \label{sec:action}

The affine group $\Aff(X)$ acts on the space $\Cov_G$
of equivalence classes of normal coverings of the Chamanara surface $X$ with deck group
$G$ via $[\p] \mapsto f\cdot [\p] = [f\circ \p]$.  As before, we identify the
affine group $\Aff(X)$ with the Veech group $\Gamma(X)$ via the derivative map
$D\colon \Aff(X) \to \Gamma(X)$ and obtain an action of~$\Gamma(X)$ on the
space~$\Cov_G$. As recorded in \cref{VGasStabiliser}, it follows from standard covering theory that  for a covering
\(p\colon Y \to X\), the Veech group \(\Gamma(Y)\) is the stabilizer of this
action. In this section, we explicitly compute the action
of $\Gamma(X)$ on $\Cov_G$ in the case that the group $G$ is~abelian.

\begin{lem}[Veech group as stabilizer of the action on $\Cov_G$]\label{VGasStabiliser}
	Let \(G\) be a group and  \(p \colon Y \to X\) a  covering of \(X\)
	with deck group~\(G\). Then the Veech group \(\Gamma(Y)\) is the stabilizer
	of the equivalence class \([p] \in\Cov_G\), that is,
	\[
		\Gamma(Y) = \{A \in \Gamma(X)|\; A \cdot [p] = [p]\}.
	\]
\end{lem}

\begin{proof}
    For \(A \in \Gamma(X)\), we have the following chain of equivalences: By
    \cref{prop:VeechGroupOfCovering}, \( A\) lies in~\(\Gamma(Y)\) if and only
    if the corresponding affine homeomorphism \(f_A\) of $X$ lifts via \(p\). This
    means, by definition, that there exists a homeomorphism \(h\) of $Y$ such that \(p
    \circ h = f_A \circ p\). But this is equivalent to the fact that~\([p] =
    [f_A \circ p]\). Since, again by definition, \(A\cdot [p] = [f_A \circ p]
    \), we are~done.
\end{proof}

Assume from now on that $G$ is abelian.
We compute the action of the generators \(P_1\),\(P_2\), \(-\Id\), and of the hyperbolic element \(H\) introduced in \cref{sec:chamanara}.

We always use the convention that a sum from $j=1$ to $0$ is empty, hence equal to $0$.

\begin{prop}[Action of $P_1$, $P_2$, $-\Id$, and $H$]\label{ComputeTheGenerators}
	Let \(p \colon Y \to X\) be a finite abelian covering of \(X\) with deck group \(G\).
	Let \(h\) be
	a defining vector of~\(\p\) and \([h]\) its equivalence class.
	The generators \(P_1\), \(P_2\) and \(-\Id\), and the element \(H\) act on~$[h]$ as follows:
	
	\begin{align*}
		(P_1\cdot h)_k &= \begin{cases}
			h_{-k} + 2\sum_{j = 1}^{k-1} (-h_j + h_{-j}), & \text{if } k \geq 1,\\
			h_{-k} + 2\sum_{j = 1}^{-k} (-h_j + h_{-j}), & \text{if } k \leq -1,
		\end{cases}\\
		(P_2\cdot h)_k &= \begin{cases}
			h_{-1} + h_{-k-1} + 2h_{-k} + 2\sum_{j = 1}^{k-1} (-h_j + h_{-j}), & \text{if } k \geq 1,\\
			h_{-1}, & \text{if } k = -1,\\
			h_{-1} + h_{-k-1} + 2h_{k} + 2\sum_{j = 1}^{-k-1} (-h_j + h_{-j}), & \text{if } k \leq -2,
		\end{cases}\\
		(-\Id\cdot h)_k &= -h_k,\\
		(H\cdot h)_k &= \begin{cases}
			h_{k-1} + h_{-1}, & \text{if } k \neq 1,\\
			-h_{-1}, & \text{if } k = 1,
		\end{cases}
	\end{align*}
where \(P_1\cdot h\),  \(P_2\cdot h\),\(-\Id\cdot h\), and \(H\cdot h\)
denote representatives of the classes \(P_1\cdot [h]\),  \(P_2\cdot
[h]\), \(-\Id\cdot [h]\), and~\(H\cdot [h]\).
\end{prop}

To prove \cref{ComputeTheGenerators}, we need more preparation. We first describe how $\Aff(X)$ acts on the defining group
homomorphism $\kappa_{\p}\colon \pi_1(X)\to G$ associated to \(p\). This
inconspicuous fact will later be essential for the computations.

 \begin{lem}[Computation of action of $\Aff(X)$ on defining group homomorphism]\label{ActOnKappa}
   	We choose $x_0 \in X$ and a preimage $y_0 \in \p^{-1}(x_0)$ as base points
	of $X$ and $Y$. For $f \in \Aff(X)$, let $f_\star\colon \pi_1(X,x_0) \to
	\pi_1(X,f(x_0))$ be the induced isomorphism of the fundamental group. Then,
	$\kappa_\p \circ {f_\star}^{-1}$ is a defining group homomorphism for~$f
	\circ \p$.
 \end{lem}

\begin{proof}
	Observe that $\Deck(f \circ \p) = \Deck(\p) = G$. We will show that
	$\kappa_\p \circ {f_\star}^{-1}$ satisfies~\eqref{monodromy} for the
	covering~$f \circ \p$. 
	
	Let $\gamma \in \pi_1(X,x_0)$ and $\gamma' = f_\star (\gamma) \in \pi_1(X,f(x_0))$.
	From the definition of $\kappa_{f \circ p}$ in \eqref{kappap}, we have that $\kappa_{f
		\circ \p}(\gamma')(y_0) = y_0 \cdot \gamma'$ which is the end point of a lift of  $\gamma'$ via~$f \circ \p$ starting at $y_0$.
	However, this lift is equal to the lift of~$\gamma = f_\star^{-1}(\gamma')$ via~$\p$ starting at $y_0$. Applying \eqref{kappap} to $\kappa_p$ gives that its end point is equal to
	$\kappa_{\p}({f_\star}^{-1}(\gamma'))(y_0 )$. Since~\eqref{monodromy} determines
	the equivalence class of the defining group homomorphism $\kappa_{f \circ
	\p}$ uniquely, we obtain the claim.
\end{proof}

Since the group $G$ is abelian, any map $\kappa\colon \pi_1(X) \to G$ factors
over the projection $q\colon \pi_1(X) \to H_1(X,\ZZ) \cong
\pi_1^{\text{ab}}(X)$. This
means that there exists some $\overline{\kappa}\colon H_1(X,\ZZ) \to G$ such
that the following diagram commutes:
\[
 \begin{xy}
 \xymatrix{
    \pi_1(X) \ar[d]^q \ar[dr]^{\kappa}    &    \\
     H_1(X,\ZZ) \ar[r]^{\overline{\kappa}}             &   G
 }
 \end{xy}
\]
For $f \in \Aff(X)$, we denote by $\overline{f_\star}\colon H_1(X,\ZZ) \to
H_1(X,\ZZ)$ the automorphism on homology induced by $f$. It follows that
$\kappa_{f \circ \p}$, which is equal to $\kappa_{\p} \circ {f_\star}^{-1}$
by \cref{ActOnKappa}, descends to $\overline{\kappa_{f \circ \p}} =
\overline{\kappa_\p} \circ \overline{f_\star}^{-1}$.\\

To simplify notation,  we denote for any $\gamma \in \pi_1(X)$ its image $q(\gamma)$ in $H_1(X,\ZZ)$ also just by
$\gamma$ from now on. Moreover, for~$f \in \Aff(X)$, we denote the
isomorphism induced by $f$ on $H_1(X,\ZZ)$ also by $f_\star$.
With this notation, we summarize the previous discussion in the following remark.

\begin{rem}[Computation of action of $\Aff(X)$ on defining vector]\label{tildeh}
We obtain a defining vector~$\tilde{h}$ of $[f\circ \p]$ as follows:
\begin{equation}\label{ActOnMondromy}
  \tilde{h}_i = \begin{cases}
		\kappa_{f \circ \p}(\alpha_i) = \overline{\kappa_{f \circ \p}}(\alpha_i) = \overline{\kappa_{\p}}({f_\star}^{-1}(\alpha_i)), & \text{if } i > 0,\\
		0, & \text{if } i = 0,\\
		\kappa_{f \circ \p}(\beta_{-i}) = \overline{\kappa_{f \circ \p}}(\beta_{-i}) = \overline{\kappa_{\p}}({f_\star}^{-1}(\beta_{-i})), & \text{if } i < 0.
  \end{cases}
\end{equation}
\end{rem}

A crucial ingredient in the proof of \cref{ComputeTheGenerators} is that \(P_1\) and \(P_2\) act as Dehn multitwists.
Recall that for a simple \emph{right} Dehn twist $f$ about a curve~$c$ and for any $\gamma \in H_1(X,\ZZ)$, we have that
\begin{equation}\label{ActionOfDehnTwist}
    f_{\star}(\gamma) = \gamma + \intersection(c,\gamma)\cdot c,
\end{equation}
where $\intersection(c_1, c_2)$ is the algebraic intersection number of two simple
closed curves $c_1$ and $c_2$. If~$f$ is a Dehn multitwist, we can decompose it as
product of Dehn twists about simple curves and then use the additivity. It is a subtle technicality that \(P_1\) and \(P_2\) are both products of
\textit{left} Dehn twists. Hence we can apply \eqref{ActionOfDehnTwist} to their
inverses.

For the calculation, we need the intersection numbers for the generating cycles
\(\alpha_i\) and \(\beta_i\) which we read off from
\cref{fig:chamanara_generators} on \cpageref{fig:chamanara_generators}, as
follows for $j,k \in \NN$:
\begin{equation}\label{isn}
    i(\alpha_j,\alpha_k) = -1, \text{ if } j < k,
    \quad i(\alpha_j, \beta_k) = -1 , \text{ for all } j,k,
    \quad \text{ and } i(\beta_j,\beta_k) = 1 \text{ if } j < k.
\end{equation}

We are now ready to prove \cref{ComputeTheGenerators}.

\begin{proof}[Proof of \cref{ComputeTheGenerators}]
	Let $\p \colon Y \to X$ be an abelian covering with deck group $G$ and defining vector~\mbox{$h \in \TheSequences_G$}. In the following, we compute a defining vector $\tilde{h}$ of $[f\circ \p]$ for $f \in	\{P_1,P_2,-\Id,H\}$, using \cref{tildeh}.

	$\bullet$ \textbf{The generator $P_1$:}

	Recall from \cref{sec:chamanara} that $P_1$ comes from the cylinder
	decomposition of slope $1$. As all cylinders have the same modulus, $P_1$
	acts as product of the simple Dehn twists about the core curves of the
	cylinders of slope $1$.
  	We see from \cref{fig:cylinder_decomposition_chamanara_slope} that the core curve $c_j$ of the~$j$--th cylinder in
          direction~$\left(
	  \begin{smallmatrix} 1 \\ 1 \end{smallmatrix}\right)$ crosses the edge $a_j$ from left to right and the edge $b_j$ from bottom to top.
          It is thus homologous to the composition of $\alpha_j$ and $\beta_j^{-1}$ (cf.~\cref{fig:chamanara_generators}). Thus we have for the corresponding element in homology, also denoted by $c_j$, the identity $c_j = \alpha_j - \beta_j$.
	From this, we obtain, using~\eqref{isn}:
	\[
		i(c_j,\alpha_k) = i(\alpha_j,\alpha_k) - i(\beta_j,\alpha_k) = \begin{cases}
			 -1 - 1 = -2, & \text{if } 1 \leq j < k,\\
			 0 - 1 = -1, & \text{if } j = k,\\
			1 - 1 = 0, & \text{if } j > k,
		\end{cases}
	\]
	and
	\[
		i(c_j, \beta_k) = i(\alpha_j,\beta_k) - i(\beta_j,\beta_k) = \begin{cases}
			-1 - 1 = -2, & \text{if } 1 \leq j < k,\\
			-1 + 0 = -1, & \text{if } j = k,\\
			-1 + 1 = 0, & \text{if } j > k.
		\end{cases}
	\]
	Applying \eqref{ActionOfDehnTwist}, we obtain for the right Dehn multitwist
	$P_1^{-1}$ about the core curves of the cylinders in direction~$\left(
	\begin{smallmatrix} 1 \\ 1 \end{smallmatrix}\right)$ that
	\[
	  (P_1^{-1})_{\star}(\gamma) = \gamma + \sum_{j \geq 1}i(c_j,\gamma)\cdot c_j.
	\]
	This gives
	\begin{equation*}
		(P_1^{-1})_{\star}(\alpha_k)
		= \alpha_k + \sum_{j \geq 1}i(c_j,\alpha_k)\cdot c_j
		=  \alpha_k - (\alpha_k - \beta_k) - 2\sum_{j = 1}^{k-1} (\alpha_j - \beta_j)
		= \beta_k + 2\sum_{j = 1}^{k-1} (-\alpha_j+\beta_j),
	\end{equation*}
	and
	\begin{align*}
			(P_1^{-1})_{\star}(\beta_k)
			& = \beta_k + \sum_{j \geq 1}i(c_j,\beta_k)\cdot c_j
			=  \beta_k - (\alpha_k - \beta_k) - 2\sum_{j = 1}^{k-1} (\alpha_j - \beta_j)\\
			& =   -\alpha_k + 2\beta_k +2\sum_{j = 1}^{k-1} (-\alpha_j+\beta_j).
	\end{align*}
	Using the definition in~\eqref{Sequence}, we obtain from~\eqref{ActOnMondromy} for $k\geq 1$
	\begin{equation*}
		(P_1\cdot h)_k = \overline{\kappa}\left( (P_1^{-1})_{\star}(\alpha_k) \right)
			= \overline{\kappa} \left( \beta_k + 2\sum_{j = 1}^{k-1} (-\alpha_j+\beta_j) \right)
			= h_{-k} + 2\sum_{j = 1}^{k-1} (-h_j+h_{-j}),
	\end{equation*}
	and for $k\leq -1$
	\begin{equation*}
        (P_1\cdot h)_k =  \overline{\kappa} \left( (P_1^{-1})_{\star}(\beta_{-k}) \right)
			= -h_{-k} + 2h_{k} + 2\sum_{j=1}^{-k-1} (-h_j + h_{-j})
			= h_{-k} + 2\sum_{j = 1}^{-k} (-h_j + h_{-j}).
	\end{equation*}

	$\bullet$ \textbf{The generator $P_2$:}

	Recall from \cref{sec:chamanara} that $P_2$ comes from the cylinder decomposition of slope $2$ where the modulus of the middle cylinder is three times the modulus of all the other cylinders. Geometrically, this implies that $P_2$ is the product of the triple
	left Dehn twist about the core curve of the middle cylinder and the simple
	left Dehn twists about the core curves of the other cylinders of slope~$2$.
	The core curves $d_j$ of the $j$--th cylinder in direction
	$\left(\begin{smallmatrix} 1 \\
	2 \end{smallmatrix}\right)$ are given as
	\begin{equation*}
		d_1 = -\beta_1, \quad d_j = \alpha_{j-1} - \beta_j \text{ for } j \geq 2
		.
	\end{equation*}
	For the intersection numbers, we obtain from \cref{isn} similarly to the previous computation
	\[
		i(d_j,\alpha_k) = \begin{cases}
			-1, & \text{if } j = 1,\\
			-1 - 1 = -2, & \text{if } 1 < j < k + 1,\\
			0 - 1 = -1, & \text{if } j = k+1,\\
			1 - 1 = 0, & \text{if } j > k+1,
		\end{cases}
	\]
	and
	\[
		i(d_1,\beta_k) = \begin{cases}
			0, & \text{if } k = 1,\\
			-1, & \text{if } k > 1,
		\end{cases}
		\qquad
		i(d_j,\beta_k) = \begin{cases}
			-1 - 1 = -2, & \text{if } 2 \leq j < k,\\
			-1 + 0 = -1, & \text{if } 2 \leq j = k,\\
			-1 + 1 = 0, & \text{if } j > k.
		\end{cases}
	\]
	Taking into account that $P_2^{-1}$ is a triple right Dehn twist about $d_1$, we obtain
	\[
		\begin{split}
			(P_2^{-1})_{\star}(\alpha_k) &= \alpha_k + 3 \cdot i(d_1, \alpha_k) \cdot d_1 + \sum_{j \geq 2}i(d_j,\alpha_k)\cdot d_j \\
        		&= \alpha_k - 3d_1 - 2\sum_{j = 2}^{k} d_j - d_{k+1}\\
        		&= \alpha_k - 3 \cdot (-\beta_1) - 2\sum_{j = 2}^{k} (\alpha_{j-1}-\beta_j) - (\alpha_k - \beta_{k+1})\\
        		&= -2\sum_{j = 1}^{k-1} \alpha_{j} + 2\sum_{j = 1}^{k} \beta_j + \beta_1 + \beta_{k+1}.
		\end{split}
	\]
	Moreover, we have
	\[
		(P_2^{-1})_{\star}(\beta_k) = \beta_k + 3 \cdot i(d_1, \beta_k) \cdot d_1 + \sum_{j \geq 2}i(d_j,\beta_k)\cdot d_j,
	\]
	which is equal to $\beta_1$ if $k = 1$ and to
	\[
		\beta_k - 3 \cdot (-\beta_1)  -2\sum_{j = 2}^{k-1} (\alpha_{j-1} -\beta_j) - (\alpha_{k-1} - \beta_k)
                        = -2\sum_{j = 1}^{k-2} \alpha_{j} - \alpha_{k-1} + \beta_1 + 2\sum_{j = 1}^{k} \beta_{j}
	\]
	if $k > 1$.
	Using the definition in \eqref{Sequence}, we obtain from~\eqref{ActOnMondromy} that $(P_2\cdot h)_{-1} =  \overline{\kappa}((P_2^{-1})_{\star}(\beta_1) ) =
	h_{-1}$, while, if $k\ge 1$:
	\[
		\begin{split}
			(P_2\cdot h)_k &= \overline{\kappa}((P_2^{-1})_{\star}(\alpha_k) )\\
			&= \overline{\kappa} \Biggl( -2\sum_{j = 1}^{k-1} \alpha_{j} + 2\sum_{j = 1}^{k} \beta_j + \beta_1 + \beta_{k+1} \Biggr) \\
        	&=  2\sum_{j = 1}^{k-1} (-h_j+h_{-j}) + h_{-1} + 2h_{-k} + h_{-(k+1)},
		\end{split}
	\]
	and if $k\le -2$,
	\[
	\begin{split}
			(P_2\cdot h)_k &=  \overline{\kappa}((P_2^{-1})_{\star}(\beta_{-k}) )\\
			&= -2\sum_{j=1}^{-k-2} h_j - h_{-k-1} + h_{-1} + 2\sum_{j=1}^{-k}h_{-j}\\
			&= 2\sum_{j=1}^{-k-1} (-h_j + h_{-j}) + h_{-k-1} + h_{-1} + 2 h_{k}.
	\end{split}
	\]

	$\bullet$ \textbf{The generator $-\Id$:}

	We easily see that
	\[
	    (-\Id^{-1})_{\star}(\alpha_k) = -\alpha_k
	    \qquad \text{and} \qquad
        (-\Id^{-1})_{\star}(\beta_k) = -\beta_k
        .
	\]
	Hence we obtain for the action:
	\[
        (-\Id\cdot h)_k = -h_k.
	\]

	$\bullet$ \textbf{The hyperbolic element $H$:}

	\begin{figure}[btph]
		\centering
		\begin{tikzpicture}[x=1cm,y=1cm, scale=0.45]

		\draw (-3,4) node {$\begin{pmatrix}
		\frac12 & 0 \\
		0 & 2
		\end{pmatrix} \cdot$};

		\draw (0,0) -- node[below] {$b_1$}
			(4,0) -- node[below] {$b_2$}
			(6,0) -- node[below] {$b_3$}
			(7,0) -- node[below] {$b_4$}
			(7.5,0) -- (7.75,0);
		\draw[densely dotted] (7.75,0) -- (8,0);

		\fill (0,0) circle (2pt);
		\fill (8,8) circle (2pt);

		\fill (4,0) circle (2pt);
		\fill (6,0) circle (2pt);
		\fill (7,0) circle (2pt);
		\fill (7.5,0) circle (2pt);
		\fill (7.75,0) circle (2pt);

		\draw (8,8) -- node[above] {$b_1$}
			(4,8) -- node[above] {$b_2$}
			(2,8) -- node[above] {$b_3$}
			(1,8) -- node[above] {$b_4$}
			(0.5,8) -- (0.25,8);
		\draw[densely dotted] (0.25,8) -- (0,8);

		\fill (4,8) circle (2pt);
		\fill (2,8) circle (2pt);
		\fill (1,8) circle (2pt);
		\fill (0.5,8) circle (2pt);
		\fill (0.25,8) circle (2pt);

		\draw (0,0) -- node[left] {$a_1$}
			(0,4) -- node[left] {$a_2$}
			(0,6) -- node[left] {$a_3$}
			(0,7) -- node[left] {$a_4$}
			(0,7.5) -- (0,7.75);
		\draw[densely dotted] (0,7.75) -- (0,8);

		\fill (0,4) circle (2pt);
		\fill (0,6) circle (2pt);
		\fill (0,7) circle (2pt);
		\fill (0,7.5) circle (2pt);
		\fill (0,7.75) circle (2pt);

		\draw (8,8) -- node[right] {$a_1$}
			(8,4) -- node[right] {$a_2$}
			(8,2) -- node[right] {$a_3$}
			(8,1) -- node[right] {$a_4$}
			(8,0.5) -- (8,0.25);
		\draw[densely dotted] (8,0.25) -- (8,0);

		\fill (8,4) circle (2pt);
		\fill (8,2) circle (2pt);
		\fill (8,1) circle (2pt);
		\fill (8,0.5) circle (2pt);
		\fill (8,0.25) circle (2pt);

		\draw (10,4) node{$=$};

		\begin{scope}[xshift=12cm]
		\begin{scope}[xscale=0.5, yscale=2]
		\draw (0,0) -- node[below] {$b_1$}
			(4,0) -- node[below] {$b_2$}
			(6,0) -- node[below] {$b_3$}
			(7,0) -- % node[below] {$b_4$}
			(7.5,0) -- (7.75,0);
		\draw[densely dotted] (7.75,0) -- (8,0);

		\draw (8,8) -- node[above] {$b_1$}
			(4,8) -- node[above] {$b_2$}
			(2,8) -- node[above] {$b_3$}
			(1,8) -- % node[above] {$b_4$}
			(0.5,8) -- (0.25,8);
		\draw[densely dotted] (0.25,8) -- (0,8);

		\draw (0,0) -- node[left] {$a_1$}
			(0,4) -- node[left] {$a_2$}
			(0,6) -- node[left] {$a_3$}
			(0,7) -- node[left] {$a_4$}
			(0,7.5) -- (0,7.75);
		\draw[densely dotted] (0,7.75) -- (0,8);

		\draw (8,8) -- node[right] {$a_1$}
			(8,4) -- node[right] {$a_2$}
			(8,2) -- node[right] {$a_3$}
			(8,1) -- node[right] {$a_4$}
			(8,0.5) -- (8,0.25);
		\draw[densely dotted] (8,0.25) -- (8,0);
		\end{scope}
		
		\fill (0,0) circle (2pt);
		\fill (4,16) circle (2pt);
		
		\fill (2,0) circle (2pt);
		\fill (3,0) circle (2pt);
		\fill (3.5,0) circle (2pt);
		\fill (3.75,0) circle (2pt);
		\fill (3.875,0) circle (2pt);
		
		\fill (2,16) circle (2pt);
		\fill (1,16) circle (2pt);
		\fill (0.5,16) circle (2pt);
		\fill (0.25,16) circle (2pt);
		\fill (0.125,16) circle (2pt);
		
		\fill (0,8) circle (2pt);
		\fill (0,12) circle (2pt);
		\fill (0,14) circle (2pt);
		\fill (0,15) circle (2pt);
		\fill (0,15.5) circle (2pt);
		
		\fill (4,8) circle (2pt);
		\fill (4,4) circle (2pt);
		\fill (4,2) circle (2pt);
		\fill (4,1) circle (2pt);
		\fill (4,0.5) circle (2pt);
		\end{scope}

		\draw (18,4) node {$\leadsto$};

		\begin{scope}[xshift=20cm]
		\draw[dashed] (0,0) -- (4,0);
		\draw (4,0) -- node[below] {$b_1$}
			(6,0) -- node[below] {$b_2$}
			(7,0) -- node[below] {$b_3$}
			(7.5,0) -- (7.75,0);
		\draw[densely dotted] (7.75,0) -- (8,0);

		\fill (0,0) circle (2pt);
		\fill (8,8) circle (2pt);

		\fill (4,0) circle (2pt);
		\fill (6,0) circle (2pt);
		\fill (7,0) circle (2pt);
		\fill (7.5,0) circle (2pt);
		\fill (7.75,0) circle (2pt);

		\draw[dashed] (8,8) -- (4,8);
		\draw (4,8) -- node[above] {$b_1$}
			(2,8) -- node[above] {$b_2$}
			(1,8) -- node[above] {$b_3$}
			(0.5,8) -- (0.25,8);
		\draw[densely dotted] (0.25,8) -- (0,8);

		\fill (4,8) circle (2pt);
		\fill (2,8) circle (2pt);
		\fill (1,8) circle (2pt);
		\fill (0.5,8) circle (2pt);
		\fill (0.25,8) circle (2pt);

		\draw (0,0) -- node[left] {$a_2$}
			(0,4) -- node[left] {$a_3$}
			(0,6) -- node[left] {$a_4$}
			(0,7) -- % node[left] {$a_5$}
			(0,7.5) -- (0,7.75);
		\draw[densely dotted] (0,7.75) -- (0,8);

		\fill (0,4) circle (2pt);
		\fill (0,6) circle (2pt);
		\fill (0,7) circle (2pt);
		\fill (0,7.5) circle (2pt);
		\fill (0,7.75) circle (2pt);

		\draw (8,8) -- node[right] {$a_2$}
			(8,4) -- node[right] {$a_3$}
			(8,2) -- node[right] {$a_4$}
			(8,1) -- % node[right] {$a_5$}
			(8,0.5) -- (8,0.25);
		\draw[densely dotted] (8,0.25) -- (8,0);

		\fill (8,4) circle (2pt);
		\fill (8,2) circle (2pt);
		\fill (8,1) circle (2pt);
		\fill (8,0.5) circle (2pt);
		\fill (8,0.25) circle (2pt);

		\draw[ultra thin] (4,0) -- node{$a_1$} (4,8);
		\end{scope}
		\end{tikzpicture}
		\caption{The action of $H^{-1}$ on the Chamanara surface.}
		\label{fig:action_H-invers}
	\end{figure}

	Finally, we determine the action of the hyperbolic element \(H\). Observe
	that we could obtain this from concatenating the action of $-\Id$, $P_1$,
	and $P_2$. However, it is convenient to compute~it directly as the action of $H^{-1} =
	\begin{pmatrix}
	\frac{1}{2} & 0 \\
	0 & 2 \end{pmatrix}$
	on the Chamanara surface can be understood geometrically very explicitly (see \cref{fig:action_H-invers}).
	Hence, we can read the action of $H^{-1}$ on the curves~$\alpha_k$ and $\beta_k$ directly from the description as a
	square with edges identified (see \cref{fig:action_H-invers_on_generators} for some examples):
	\[
		(H^{-1})_{\star}(\alpha_k) = \begin{cases}
			- \beta_1, & \text{if } k = 1,\\
			\alpha_{k-1} + \beta_1, & \text{if } k > 1,
		\end{cases}
	\]
	and $ (H^{-1})_{\star}(\beta_k) = \beta_{k+1} + \beta_1$. Hence, we obtain
	\[
		(H\cdot h)_k = \begin{cases}
			h_{k-1} + h_{-1}, & \text{if } k \neq 1, \\
			- h_{-1}, & \text{if } k = 1.
		\end{cases}\qedhere
	\]
    
    	\begin{figure}
    	\centering
    	\begin{tikzpicture}[x=1cm,y=1cm, scale=0.45,
    		mid arrow/.style={
    			postaction={decorate,decoration={
    					markings,
    					mark=at position .66 with {\arrow{latex}}
    			}}
    		}]
    		
    		\draw[dashed] (0,0) -- (4,0);
    		\draw (4,0) -- node[below] {$b_1$}
    		(6,0) -- node[below] {$b_2$}
    		(7,0) -- node[below] {$b_3$}
    		(7.5,0) -- (7.75,0);
    		\draw[densely dotted] (7.75,0) -- (8,0);
    		
    		\fill (0,0) circle (2pt);
    		\fill (8,8) circle (2pt);
    		
    		\fill (4,0) circle (2pt);
    		\fill (6,0) circle (2pt);
    		\fill (7,0) circle (2pt);
    		\fill (7.5,0) circle (2pt);
    		\fill (7.75,0) circle (2pt);
    		
    		\draw[dashed] (8,8) -- (4,8);
    		\draw (4,8) -- node[above] {$b_1$}
    		(2,8) -- node[above] {$b_2$}
    		(1,8) -- node[above] {$b_3$}
    		(0.5,8) -- (0.25,8);
    		\draw[densely dotted] (0.25,8) -- (0,8);
    		
    		\fill (4,8) circle (2pt);
    		\fill (2,8) circle (2pt);
    		\fill (1,8) circle (2pt);
    		\fill (0.5,8) circle (2pt);
    		\fill (0.25,8) circle (2pt);
    		
    		\draw (0,0) -- node[left] {$a_2$}
    		(0,4) -- node[left] {$a_3$}
    		(0,6) -- node[left] {$a_4$}
    		(0,7) -- % node[left] {$a_5$}
    		(0,7.5) -- (0,7.75);
    		\draw[densely dotted] (0,7.75) -- (0,8);
    		
    		\fill (0,4) circle (2pt);
    		\fill (0,6) circle (2pt);
    		\fill (0,7) circle (2pt);
    		\fill (0,7.5) circle (2pt);
    		\fill (0,7.75) circle (2pt);
    		
    		\draw (8,8) -- node[right] {$a_2$}
    		(8,4) -- node[right] {$a_3$}
    		(8,2) -- node[right] {$a_4$}
    		(8,1) -- % node[right] {$a_5$}
    		(8,0.5) -- (8,0.25);
    		\draw[densely dotted] (8,0.25) -- (8,0);
    		
    		\fill (8,4) circle (2pt);
    		\fill (8,2) circle (2pt);
    		\fill (8,1) circle (2pt);
    		\fill (8,0.5) circle (2pt);
    		\fill (8,0.25) circle (2pt);
    		
    		\draw[ultra thin] (4,0) -- node{$a_1$} (4,8);
    		
    		\fill (2,1.5) circle (3pt);
    		\draw[thick, mid arrow] (2,1.5) .. controls +(-130:0.5cm) and +(90:0.5cm) .. (1.5,0);
    		\draw[thick, mid arrow] (5.5,8) .. controls +(-90:0.7cm) and +(80:0.7cm) .. (5.2,2) node[] {$H^{-1}(\beta_1)$} .. controls +(-105:0.7cm) and +(90:1cm) .. (5,0);
    		\draw[thick] (3,8) .. controls +(-90:1cm) and +(50:1cm) .. (2,1.5);

    		\begin{scope}[xshift=12cm]
    			\draw[dashed] (0,0) -- (4,0);
    			\draw (4,0) -- node[below] {$b_1$}
    			(6,0) -- node[below] {$b_2$}
    			(7,0) -- node[below] {$b_3$}
    			(7.5,0) -- (7.75,0);
    			\draw[densely dotted] (7.75,0) -- (8,0);
    			
    			\fill (0,0) circle (2pt);
    			\fill (8,8) circle (2pt);
    			
    			\fill (4,0) circle (2pt);
    			\fill (6,0) circle (2pt);
    			\fill (7,0) circle (2pt);
    			\fill (7.5,0) circle (2pt);
    			\fill (7.75,0) circle (2pt);
    			
    			\draw[dashed] (8,8) -- (4,8);
    			\draw (4,8) -- node[above] {$b_1$}
    			(2,8) -- node[above] {$b_2$}
    			(1,8) -- node[above] {$b_3$}
    			(0.5,8) -- (0.25,8);
    			\draw[densely dotted] (0.25,8) -- (0,8);
    			
    			\fill (4,8) circle (2pt);
    			\fill (2,8) circle (2pt);
    			\fill (1,8) circle (2pt);
    			\fill (0.5,8) circle (2pt);
    			\fill (0.25,8) circle (2pt);
    			
    			\draw (0,0) -- node[left] {$a_2$}
    			(0,4) -- node[left] {$a_3$}
    			(0,6) -- node[left] {$a_4$}
    			(0,7) -- % node[left] {$a_5$}
    			(0,7.5) -- (0,7.75);
    			\draw[densely dotted] (0,7.75) -- (0,8);
    			
    			\fill (0,4) circle (2pt);
    			\fill (0,6) circle (2pt);
    			\fill (0,7) circle (2pt);
    			\fill (0,7.5) circle (2pt);
    			\fill (0,7.75) circle (2pt);
    			
    			\draw (8,8) -- node[right] {$a_2$}
    			(8,4) -- node[right] {$a_3$}
    			(8,2) -- node[right] {$a_4$}
    			(8,1) -- % node[right] {$a_5$}
    			(8,0.5) -- (8,0.25);
    			\draw[densely dotted] (8,0.25) -- (8,0);
    			
    			\fill (8,4) circle (2pt);
    			\fill (8,2) circle (2pt);
    			\fill (8,1) circle (2pt);
    			\fill (8,0.5) circle (2pt);
    			\fill (8,0.25) circle (2pt);
    			
    			\draw[ultra thin] (4,0) -- node{$a_1$} (4,8);
    			
    			\fill (2,1.5) circle (3pt);
    			\draw[thick, mid arrow] (2,1.5) .. controls +(40:1cm) and +(-120:1cm) .. (4.3,3) node[right] {$H^{-1}(\alpha_1)$} .. controls +(60:1cm) and +(-90:0.5cm) .. (5,8);
    			\draw[thick] (1,0) .. controls +(90:0.5cm) and +(-140:0.5cm) .. (2,1.5);
    		\end{scope}
    		
    		\begin{scope}[xshift=24cm]
    			\draw[dashed] (0,0) -- (4,0);
    			\draw (4,0) -- node[below] {$b_1$}
    			(6,0) -- node[below] {$b_2$}
    			(7,0) -- node[below] {$b_3$}
    			(7.5,0) -- (7.75,0);
    			\draw[densely dotted] (7.75,0) -- (8,0);
    			
    			\fill (0,0) circle (2pt);
    			\fill (8,8) circle (2pt);
    			
    			\fill (4,0) circle (2pt);
    			\fill (6,0) circle (2pt);
    			\fill (7,0) circle (2pt);
    			\fill (7.5,0) circle (2pt);
    			\fill (7.75,0) circle (2pt);
    			
    			\draw[dashed] (8,8) -- (4,8);
    			\draw (4,8) -- node[above] {$b_1$}
    			(2,8) -- node[above] {$b_2$}
    			(1,8) -- node[above] {$b_3$}
    			(0.5,8) -- (0.25,8);
    			\draw[densely dotted] (0.25,8) -- (0,8);
    			
    			\fill (4,8) circle (2pt);
    			\fill (2,8) circle (2pt);
    			\fill (1,8) circle (2pt);
    			\fill (0.5,8) circle (2pt);
    			\fill (0.25,8) circle (2pt);
    			
    			\draw (0,0) -- node[left] {$a_2$}
    			(0,4) -- node[left] {$a_3$}
    			(0,6) -- node[left] {$a_4$}
    			(0,7) -- % node[left] {$a_5$}
    			(0,7.5) -- (0,7.75);
    			\draw[densely dotted] (0,7.75) -- (0,8);
    			
    			\fill (0,4) circle (2pt);
    			\fill (0,6) circle (2pt);
    			\fill (0,7) circle (2pt);
    			\fill (0,7.5) circle (2pt);
    			\fill (0,7.75) circle (2pt);
    			
    			\draw (8,8) -- node[right] {$a_2$}
    			(8,4) -- node[right] {$a_3$}
    			(8,2) -- node[right] {$a_4$}
    			(8,1) -- % node[right] {$a_5$}
    			(8,0.5) -- (8,0.25);
    			\draw[densely dotted] (8,0.25) -- (8,0);
    			
    			\fill (8,4) circle (2pt);
    			\fill (8,2) circle (2pt);
    			\fill (8,1) circle (2pt);
    			\fill (8,0.5) circle (2pt);
    			\fill (8,0.25) circle (2pt);
    			
    			\draw[ultra thin] (4,0) -- node{$a_1$} (4,8);
    			
    			\fill (2,1.5) circle (3pt);
    			\draw[thick, mid arrow] (2,1.5) .. controls +(-45:0.7cm) and +(90:0.5cm) .. (2.5,0);
    			\draw[thick, mid arrow] (6.5,8) .. controls +(-90:0.7cm) and +(135:0.7cm) .. (7,6.5) node[left] {$H^{-1}(\alpha_2)$} .. controls +(-45:0.7cm) and +(180:0.2cm) .. (8,6);
    			\draw[thick] (0,2) .. controls +(0:0.5cm) and +(135:0.5cm) .. (2,1.5);
    		\end{scope}
    	\end{tikzpicture}
    	\caption{The images of $\beta_1$, $\alpha_1$, and $\alpha_2$ under $H^{-1}$.}
    	\label{fig:action_H-invers_on_generators}
    \end{figure}

\end{proof}

As a direct consequence of \cref{ComputeTheGenerators}, we obtain the following remark.

\begin{rem}[$-\Id$ always lifts]\label{rem:Id_lifts}
 Let $Y_h$ be a finite abelian cover. Then $h$ is a fixed point of $-\Id$ if and only if $h \sim -h$ in $\mathcal{B}_G$. This is always true as $g \mapsto -g$ is an automorphism in any abelian group.
\end{rem}

Later, we will need to understand the fixed points of the action of $H^n$. Therefore, we also compute the action of $H^n$.

\begin{prop}[Action of $H^n$] \label{prop:action_H_n}
	Let $h \in \TheSequences_G$ be a defining vector, $n\in \mathbb{N}$, and $c_n
	\coloneqq \sum\limits_{j=1}^n 2^{n-j} h_{-j}$. Then, we have:
	\begin{equation*}
		(H^n \cdot h)_k = \begin{cases}
			h_{k-n} + c_n, & \text{if } k > n \text{ or if } k \leq -1,\\
			- 3 \sum\limits_{j=1}^{n-k} 2^{n-k-j} h_{-j} - 2h_{-n+k-1}  + c_n, & \text{if } 1 \leq k \leq n.
		\end{cases}
	\end{equation*}
\end{prop}

\begin{proof}
	We prove the statement by induction on $n$. For $n = 1$, the equations are a special case of \cref{ComputeTheGenerators}. Now let $n\in \mathbb{N}$ and assume that the equations are true for $H^n$. Then we have for $k > n+1$ and for $k \leq -1$:
	\begin{align*}
		(H^{n+1}h)_k & = (H^n h)_{k-1} + (H^n h)_{-1} \\
		& = h_{k-1-n} + \sum\limits_{j=1}^n 2^{n-j} h_{-j} + h_{-1-n} + \sum\limits_{j=1}^n 2^{n-j} h_{-j} \\
		& = h_{k-(n+1)} + \sum\limits_{j=1}^{n+1} 2^{n+1-j} h_{-j}
	\end{align*}
	For $1 < k \leq n+1$:
	\begin{align*}
		(H^{n+1}h)_k & = (H^n h)_{k-1} - (H^n h)_{-1} \\
		& = - 3 \sum\limits_{j=1}^{n-k+1} 2^{n-k+1-j} h_{-j} - 2h_{-n+k-1-1} + \sum\limits_{j=1}^n 2^{n-j} h_{-j}
		 + h_{-1-n} + \sum\limits_{j=1}^n 2^{n-j} h_{-j} \\
		& = - 3 \sum\limits_{j=1}^{n+1-k} 2^{n+1-k-j} h_{-j} - 2h_{-(n+1)+k-1} + \sum\limits_{j=1}^{n+1} 2^{n+1-j} h_{-j}
	\end{align*}
	And finally for $k=1$:
	\begin{align*}
		(H^{n+1}h)_1
		= -(H^n h)_{-1}
		& = - h_{-1-n} - \sum\limits_{j=1}^n 2^{n-j} h_{-j} \\
		& = - 3 \sum\limits_{j=1}^{n} 2^{n-j} h_{-j}
		- h_{-(n+1)} + 2 \sum\limits_{j=1}^{n} 2^{n-j} h_{-j}
		\qedhere
	\end{align*}
\end{proof}

\section{Characterization of Veech groups of finite index} \label{sec:characterization_finite_index}

Let  \(G\) be a finite abelian group, \(h \in \TheSequences_G\), and \(p\colon Y_h \to X\) a covering with defining
vector \(h\). In this section, we characterize when $\Gamma(Y_h)$ has finite
index in $\Gamma(X)$.

Recall the identification of \(\Cov_G\) with \(\TheSequencesQuot_G\) given by
\cref{DescriptionBG}. By \cref{VGasStabiliser}, we have that
$\Gamma(Y_h)$ is the stabilizer of \([h]\) in \(\Gamma(X)\). Therefore
\(\Gamma(Y_h)\) has finite index in \(\Gamma(X)\) if and only if the orbit of
\([h]\) under the action of \(\Gamma(X)\) is finite. Hence, if the set
$\{H^n\cdot [h] : n \in \mathbb{N}\}$ is infinite, the index of $\Gamma(Y_h)$ is infinite as well. This set being finite, on the contrary, implies that
there exists an $n\in \mathbb{N}$ such that $[h]$ is a fixed point of $H^n$.

That finite index is actually equivalent to $h$ being a fixed point of a power of $H$ is the content of our main theorem, \cref{thm:characterization_finite_index}.
To prove the direction that being a fixed point of a power of~$H$ implies finite index of the Veech group, we consider for every $n \in \mathbb{N}$ a finite set $C_n$ that will turn out to
contain all fixed points of $H^n$ and to be invariant under the action of
$\Gamma(X)$.

\begin{definition}[The subset $C_n$]
	\label{def:C_n}
	For a finite abelian group $G$ of order $d$ and $n\in \mathbb{N}$, $C_n$ is
	defined to be the subset of $\TheSequences_G$ of bi-infinite vectors~$h$
	that fulfill the following conditions:
	\begin{itemize}
		\item \emph{$dn$--periodic backwards and forwards}: $h_k = h_{k+dn}$ and
		$h_{-k} = h_{-k-dn}$ for every $k\geq 1$.
		\item \emph{period sum is $0$}: $\sum\limits_{j=1}^{dn} (- h_j + h_{-j}) =
		0$.
		\item \emph{forward-period--backward-period relation (or forper--backper
		relation, for short) at place $k \in \{1, \ldots, dn-1\}$}: For all $k =
		1, \ldots, dn-1$:
		\begin{equation}\label{eq:forper-backper}
			2 h_{dn-k+1} - h_{dn-k} = 2 h_{-k-1} - h_{-k}.
		\end{equation}
		\item \emph{forper--backper relation at place $0$ and $dn$}:
		\begin{equation}\label{eq:forper-backper_0_dn}
			h_{dn} = - 2h_{-1} \qquad \text{and} \qquad h_{-dn} = - 2h_1.
		\end{equation}
	\end{itemize}
\end{definition}

\begin{rem}[$C_n$ is finite]
	\label{rem:C_n_finite}
	The first condition ($dn$--periodicity) implies that $C_n$ is finite for every~$n\in \mathbb{N}$. 
\end{rem}

We will now show that, for every $n\in \mathbb{N}$, any fixed point of $H^n$ is
contained in the corresponding~$C_n$.

\begin{prop}[Fixed points of $H^n$ are in $C_n$]
	\label{prop:fixed_points_H_n_in_C_n}
	Let $G$ be a finite abelian group of order $d$ and $n \in \mathbb{N}$. If $h \in \TheSequences_G$ is a fixed point of $H^n$, then we have~$h \in C_n$.

\begin{proof}
	Let $c \coloneq c_n \coloneq \sum\limits_{j=1}^n 2^{n-j} h_{-j}$. We check the
	conditions for $h$ to be an element of $C_n$:
	\begin{itemize}
		\item $h$ is $dn$--periodic forwards:
		For $k \geq 1$, iterative application of \cref{prop:action_H_n} yields
		\begin{equation*}
			h_{k+dn}
			= (H^n h)_{k+dn}
			= h_{k+(d-1)n} + c
			= \ldots
			= h_k + d \cdot c
			= h_k.
		\end{equation*}

		\item $h$ is $dn$--periodic backwards:
		For $k \geq 1$, iterative application of \cref{prop:action_H_n} yields
		\begin{equation*}
			h_{-k}
			= (H^n h)_{-k}
			= h_{-k-dn} + d \cdot c
			= h_{-k-dn}.
		\end{equation*}

		\item Period sum of $h$:
		For $k \geq n+1$, we have:
		\begin{equation*}
			-h_k + h_{-k}
			= -(H^n h)_k + h_{-k}
			= -h_{k-n} - c + (H^n h)_{-k+n} - c
			= -h_{k-n} + h_{-(k-n)} -2c.
		\end{equation*}
		Therefore, we obtain
		\begin{equation*}
			\sum\limits_{j=1}^{dn} (-h_j + h_{-j})
			= d \cdot \sum\limits_{j=1}^n (-h_j + h_{-j}) - n \cdot \sum\limits_{i=1}^{d-1} i \cdot 2 c
			= - n \cdot d(d-1) \cdot c
			= 0
			.
		\end{equation*}

		\item Forper--backper relation at place $0$:
		Iterative application of \cref{prop:action_H_n} yields
		\begin{equation*}
			h_{dn}
			= h_n + (d-1) \cdot c
			= - 2h_{-1} + c + (d-1) c = -2h_{-1}
			.
		\end{equation*}

		\item Forper--backper relation at place $dn$:
		Iterative application of \cref{prop:action_H_n} yields
		\begin{align*}
			-2h_1
			& = -2(H^n h)_1
			= 2 (H^{n-1}h)_{-1}
			= 2 \Biggl( h_{-1-(n-1)} + \sum\limits_{j=1}^{n-1} 2^{n-1-j} h_{-j} \Biggr) \\
			& = 2 h_{-n} + \sum\limits_{j=1}^{n-1} 2^{n-j} h_{-j}
			= h_{-n} + c
			= h_{-dn} + (d-1) \cdot c + c
			= h_{-dn}.
		\end{align*}
		\item
		Forper--backper relation at place $k$:
		For $k \in \{1, \ldots, n-1\}$, we have
		\begin{align*}
			& 2h_{dn-k +1} - h_{dn-k}
			= 2h_{n-k +1} + 2(d-1) c - h_{n - k} - (d -1)c \\
			& = 2 \Biggl( - 2h_{-k} + c - 3 \sum\limits_{j=1}^{k-1} 2^{k-1-j} h_{-j} \Biggr)
			 - \Biggl( - 2h_{-k-1} + c - 3 \sum\limits_{j=1}^{k} 2^{k-j} h_{-j} \Biggr) - c \\
			& = - 4 h_{-k} - 3 \sum\limits_{j=1}^{k-1} 2^{k-j} h_{-j} + 2h_{-k-1} + 3 \sum\limits_{j=1}^{k-1} 2^{k-j} h_{-j} + 3h_{-k} \\
			& = 2h_{-k-1} - h_{-k}.
		\end{align*}

		For $k=n$, we have
		\begin{align*}
			2h_{dn-n +1} - h_{dn-n}
			& = 2h_{1} + 2(d-1) c - h_{n} - (d -2)c \\
			& = 2 \Biggl( - 2h_{-n} + c - 3 \sum\limits_{j=1}^{n-1} 2^{n-1-j} h_{-j} \Biggr) - \left( - 2h_{-1} + c \right) \\
			& = - 4 h_{-n} - 3 \sum\limits_{j=1}^{n-1} 2^{n-j} h_{-j} + 2h_{-1} + c  \\
			& = - h_{-n} -3c + 2 h_{-1} + c  \\
			& = - h_{-n} + 2h_{-n-1} .
		\end{align*}
		
		And finally for $k$ with $n < k \leq dn-1$, we have $-k+n \leq -1$ and $dn +n -k > n$ and~hence
		\begin{align*}
			h_{dn-k} -h_{-k}
			& = h_{dn+n-k} - c -h_{-k+n}+c
			= h_{dn-(k-n)} - h_{-(k-n)}, \\
			2(h_{dn-k+1} -h_{-k-1})
			& = 2(h_{dn+n-k+1} - c -h_{-k+n -1}+c)
			= 2(h_{dn-(k-n) +1} - h_{-(k-n)-1})
			.
		\end{align*}
		By applying these relations successively, we can make $k$ smaller and reduce this case to the case of $k \in \{1,\ldots, n\}$ for which we already know that the forper--backper relation holds.
	\end{itemize}
	This completes the proof that $h \in C_n$.
\end{proof}
\end{prop}

We now show that for every $n\in \mathbb{N}$, the finite set $C_n$ is invariant
under the action of $\Gamma(X)$.

\begin{prop}[$C_n$ is an invariant set] \label{thm:invariant_set}
	Let $G$ be a finite abelian group, $n\in \mathbb{N}$, and $C_n$ as in \cref{def:C_n}.
	Then $C_n$ is invariant under the action of $\Gamma(X)$.
\end{prop}

\begin{proof}
	To keep the calculations as short as possible, we use the generating set $\{P_1, H, -\Id\}$ of~$\Gamma(X)$ for this proof. That means that for
	$h\in C_n$, we have to check that $P_1 \cdot h \in C_n$, $H \cdot h \in C_n$, and $-\Id \cdot h \in C_n$.

	It is clear that $P_1 h$, $H h$, and $-h$ are in $\TheSequences_G$.
	Furthermore, as all the conditions in the definition of $C_n$ are linear, we see immediately
	that $-h \in C_n$.

	We now consider all conditions for $P_1 h\in C_n$. We will constantly use the computations from \cref{ComputeTheGenerators} and that $h\in C_n$ itself.

	\begin{itemize}
		\item $P_1 h$ is $dn$--periodic forwards:

		For $k \geq 1$, using that $h$ is $dn$--periodic backwards and forwards
		and that the period sum of~$h$ is~$0$, we have that
		\[
		\begin{split}
			(P_1h)_{k+dn} - (P_1h)_k & = h_{-(k+dn)} + 2\sum_{j = 1}^{k+dn-1} (-h_j + h_{-j})  - \Biggl( h_{-k} + 2\sum_{j = 1}^{k-1} (-h_j + h_{-j}) \Biggr) \\
			& = 2 \sum_{j=dn +1}^{dn +k -1} (-h_j + h_{-j}) - 2 \sum_{j=1}^{k-1} (-h_{j+dn} + h_{-j-dn})
			= 0.
		\end{split}
		\]
		\item $P_1h$ is $dn$--periodic backwards:

		For $k \geq 1$, using that $h$ is $dn$--periodic backwards and forwards
		and that the period sum of $h$ is~$0$, we have that
		\[
		\begin{split}
			(P_1h)_{-k-dn} - (P_1h)_{-k}
			& = h_{k+dn} + 2\sum_{j = 1}^{k +dn} (-h_j + h_{-j} )  - h_{k} - 2\sum_{j=1}^{k} (-h_j + h_{-j}) \\
			& = 2 \sum_{j= dn +1}^{ dn + k} (-h_j + h_{-j}) - 2 \sum_{j=1}^{k} (-h_{j+dn} + h_{-j-dn})
			= 0.
		\end{split}
		\]

		\item Period sum of $P_1 h$:

		For $k\geq 1$, it holds that
		\begin{equation*}
			      - (P_1h)_k + (P_1h)_{-k} 
%   			      =  - h_{-k} + h_k + 2(-h_{k} + h_{-k}) 
			      =  - \Biggl( h_{-k} + 2 \sum_{j=1}^{k-1} (-h_j + h_{-j}) \Biggr) + h_k + 2 \sum_{j=1}^{k} (-h_{j} + h_{-j}) 
			      = -h_k + h_{-k}.
		\end{equation*}
		Thus, using that the period sum of $h$ is $0$, we have
		\begin{equation*}
			\sum\limits_{j=1}^{dn} \left( -(P_1 h)_j + (P_1 h)_{-j} \right)
			= \sum\limits_{j=1}^{dn} ( -h_j + h_{-j} )
			= 0.
		\end{equation*}

		\item Forper--backper relation at place $0$ for $P_1h$:

		Since the period sum of $h$ is $0$ and $h$ fulfills the forper--backper
		relation~\eqref{eq:forper-backper_0_dn} at places $0$ and~$dn$, we
		have that
		\[
		\begin{split}
			(P_1 h)_{dn} + 2 (P_1 h)_{-1}
			& = h_{-dn} + 2 \sum\limits_{j= 1}^{dn-1} (-h_j + h_{-j}) + 2( h_1 + 2(-h_{1} + h_{-1})) \\
			& = 2 \sum\limits_{j=1}^{dn} (- h_j + h_{-j}) - 2(-h_{dn} + h_{-dn}) + h_{-dn} - 2 h_1 + 4 h_{-1} = 0
			.
		\end{split}
		\]

		\item Forper--backper relation at place $dn$ for $P_1h$:

		Since the period sum of $h$ is $0$ and $h$ fulfills the forper--backper
		relation~\eqref{eq:forper-backper_0_dn} at place $0$, we have
		that
		\begin{equation*}
			(P_1 h)_{-dn} + 2 (P_1 h)_{1}
			= h_{dn} + 2 \sum\limits_{j= 1}^{dn} (-h_j + h_j) + 2 h_{-1} = 0
			.
		\end{equation*}

		\item Forper--backper relation at place $k\in \{1, \ldots, dn-1\}$ for $P_1h$:
		\begin{align*}
			& 2(P_1 h)_{dn -k +1} - (P_1 h)_{dn-k} +(P_1 h)_{-k} - 2 (P_1 h)_{-k-1} \\
			& = 2 \Biggl( h_{-dn+k-1} + 2 \sum\limits_{j=1}^{dn-k} (-h_j + h_{-j}) \Biggr)
			-  \Biggl( h_{-dn+k} + 2 \sum\limits_{j=1}^{dn - k - 1} (-h_j + h_{-j}) \Biggr) \\
			& \quad + \Biggl( h_{k} + 2 \sum\limits_{j=1}^k (-h_j + h_{-j}) \Biggr)
			- 2 \Biggl(h_{k+1} + 2 \sum\limits_{j=1}^{k+1} (-h_j + h_{-j}) \Biggr)
		\end{align*}
		Since $h$ fulfills the forper--backper relation~\eqref{eq:forper-backper} at place $dn-k$, we have
		\[
			2h_{-dn+k-1} - h_{-dn+k} + h_k -2h_{k+1} = 0.
		\]
		So we only have to consider the four sums in the computation. For the first sum, we use that the period sum of $h$ is $0$ and reverse the order of the summands in the third step to obtain
		\begin{align*}
			\sum\limits_{j=1}^{dn-k} (-h_j + h_{-j})
			& = -\sum\limits_{j=dn-k+1}^{dn} (-h_j + h_{-j})
			= \sum\limits_{j=1}^{k} (h_{dn-k+j} - h_{- dn +k -j}) , \\
			& = \sum\limits_{j=1}^{k} (h_{dn-j+1} - h_{- dn +j-1})
			= \sum\limits_{j=1}^{k} (-h_{-dn +j -1} + h_{dn -j +1})
			.
		\end{align*}
		Similarly, we have for the second sum
		\begin{equation*}
			\sum\limits_{j=1}^{dn - k - 1} (-h_j + h_{-j})
			= \sum\limits_{j=0}^{k} (-h_{-dn+j} + h_{dn-j})
			.
		\end{equation*}

		Hence, the above term for the forper--backper relation is equal to
		\begin{align*}
			& 2 \Biggl( 2 \sum\limits_{j=1}^{dn-k} (-h_j + h_{-j})
			- \sum\limits_{j=1}^{dn - k - 1} (-h_j + h_{-j})
			+ \sum\limits_{j=1}^k (-h_j + h_{-j})
			- 2 \sum\limits_{j=1}^{k+1} (-h_j + h_{-j}) \Biggr) \\
			& = 2 \Biggl( 2 \sum\limits_{j=1}^{k} (-h_{-dn +j -1} + h_{dn -j +1})
			- \sum\limits_{j=0}^{k} (-h_{-dn+j} + h_{dn-j}) \\
			& \quad + \sum\limits_{j=1}^k (-h_j + h_{-j})
			- 2 \sum\limits_{j=0}^{k} (-h_{j+1} + h_{-j-1}) \Biggr) \\
			& = 2 \biggl( h_{-dn} - h_{dn} + 2 h_1 -2 h_{-1} \\
			 & \quad + \! \sum_{j=1}^k ( 2(-h_{-dn +j -1} \! + \! h_{dn -j +1}) \! - \! (-h_{-dn+j} \! + \! h_{dn-j}) \! + \! (-h_j \! + \! h_{-j}) \! - \! 2(-h_{j+1} \! + \! h_{-j-1}) \! ) \! \!
			\biggr) \!
			.
		\end{align*}

		Finally, we use that $h$ fulfills the forper--backper relation~\eqref{eq:forper-backper} at any
		place $j \in \{1, \ldots, k\} \cup \{dn-1, \ldots, dn-k\}$ as well as at
		places $dn$ and $0$. Hence, the last line is equal to $0$.
	\end{itemize}

	Therefore, $P_1 h \in C_n$ for $h\in C_n$.
	We now show that $H h \in C_n$ as~well.
	\begin{itemize}
		\item $H h$ is $dn$--periodic forwards:
		
		For $k \geq 2$, using that $h$ is $dn$--periodic forwards, we have
		\begin{equation*}
			(H h)_{k+dn} - (H h)_k
			= h_{k+dn-1} + h_{-1} - h_{k-1} - h_{-1}
			= 0.
		\end{equation*}
		
		For $k=1$, we use that $h$ fulfills the forper--backper relation~\eqref{eq:forper-backper_0_dn} at place $0$. Hence
		\begin{equation*}
			(H h)_{1+dn} - (H h)_1
			= h_{dn} + h_{-1} - (- h_{-1})
			= -2 h_{-1} + 2h_{-1}
			= 0.
		\end{equation*}
		
		\item $H h$ is $dn$--periodic backwards:
		
		For $k \geq 1$, using that $h$ is $dn$--periodic backwards, we have
		\begin{equation*}
			(H h)_{-k-dn} - (H h)_{-k}
			= h_{-k-dn-1} + h_{-1} - h_{-k-1} - h_{-1}
			= 0.
		\end{equation*}
		
		\item Period sum of $H h$:
		
		Using that
		the period sum of $h$ is $0$,
		$h$ fulfills the forper--backper relation~\eqref{eq:forper-backper_0_dn} at place $0$,
		and that $h$ is $dn$--periodic backwards,
		we have
		\[
		\begin{split}
			\sum\limits_{j=1}^{dn} \left( -(H h)_j + (H h)_{-j} \right)
			& = - \sum\limits_{j=2}^{dn} \left( h_{j-1} + h_{-1} \right)  + h_{-1} + \sum\limits_{j=1}^{dn} \left( h_{-j-1} + h_{-1} \right) \\
			& = - \sum\limits_{j=1}^{dn-1} h_{j} + \sum\limits_{j=2}^{dn+1} h_{-j} + 2 h_{-1} \\
			& = \sum\limits_{j=1}^{dn} \left(-h_{j} +  h_{-j} \right) + h_{dn} + h_{-dn-1} + h_{-1} \\
			& = -2h_{-1} + h_{-1} + h_{-1}
			= 0.
		\end{split}
		\]
		
		\item Forper--backper relation at place $0$ for $H h$:
		
		Using that $h$ fulfills the forper--backper relation~\eqref{eq:forper-backper} at place $1$
		and the forper--backper relation~\eqref{eq:forper-backper_0_dn} at place $0$, we get:
		\begin{align*}
			(H h)_{dn} + 2 (H h)_{-1}
			& =  h_{dn-1} + h_{-1} + 2h_{-2} + 2 h_{-1} \\
			& = \left( 2h_{dn} - 2 h_{-2} + h_{-1} \right) + 2h_{-2} + 3h_{-1} \\
			& = 2 \cdot \left( -2 h_{-1} \right) + 4 h_{-1}
			= 0.
		\end{align*}
		
		\item Forper--backper relation at place $dn$ for $Hh$:
		
		Using that $h$ is $dn$--periodic backwards, we obtain
		\begin{equation*}
			(H h)_{-dn} + 2 (H h)_{1}
			= h_{-dn-1} + h_{-1} - 2h_{-1}
			= h_{-1} -h_{-1}
			= 0.
		\end{equation*}
		
		\item Forper--backper relation at place $k\in \{1, \ldots, dn-1\}$ for $Hh$:
		
		For $k \in \{1, \ldots, dn-2\}$, we have that $h$ fulfills the forper--backper relation~\eqref{eq:forper-backper} at place~$k+1$, hence
		\begin{align*}
			& 2(H h)_{dn -k +1} - (H h)_{dn-k} + (H h)_{-k} - 2 (H h)_{-k-1} \\
			& = 2 \left( h_{dn-k} + h_{-1} \right) - \left( h_{dn-k-1} + h_{-1} \right) + \left( h_{-k-1} + h_{-1} \right) - 2 \left( h_{-k-2} + h_{-1}\right) \\
			& = 2 h_{dn-k} - h_{dn-k-1} + h_{-k-1} - 2h_{-k-2}
			= 0.
		\end{align*}
		
		For $k = dn-1$, using that $h$ fulfills the forper--backper relation~\eqref{eq:forper-backper_0_dn} at place $dn$ and that~$h$ is $dn$--periodic backwards, we obtain		
		\begin{align*}
			& 2(H h)_{2} - (H h)_{1} + (H h)_{-dn+1} - 2 (H h)_{-dn} \\
			& = 2 \left( h_1 + h_{-1} \right) - \left( - h_{-1} \right) + \left( h_{-dn} + h_{-1} \right) -2 \left( h_{-dn-1} + h_{-1} \right) \\
			& = 2h_1 + 2h_{-1} -2 h_{1} -2 h_{-1}
			= 0.
		\end{align*}
	\end{itemize}
	Therefore, $H h \in C_n$.
	This finishes the proof that $C_n$ is invariant under the action of $\Gamma(X)$.
\end{proof}

Combining \cref{prop:fixed_points_H_n_in_C_n} and \cref{thm:invariant_set}, we can prove now \cref{thm:characterization_finite_index} which states that the Veech group $\Gamma(Y_h)$ of a cover $Y_h$ has finite index in $\Gamma(X)$ if and only if the defining vector $h$ of $Y_h$ is a fixed point of $H^n$ for some $n\in \mathbb{N}$.

	\begin{proof}[Proof of \cref{thm:characterization_finite_index}]
		If the index of $\Gamma(Y_h)$ in $\Gamma(X)$ is finite, then the orbit
		of $h$ under $\left\langle H \right\rangle$ has to be finite. In
		particular, there needs to exist an $n\in \mathbb{N}$ such that $H^n
		\cdot h = h$. 

		For the converse, assume that there exists an $n\in \mathbb{N}$ such that $h$ is a
		fixed point of $H^n$. Then,~$h \in C_n$ by
		\cref{prop:fixed_points_H_n_in_C_n} and $C_n$ is a finite set by \cref{rem:C_n_finite}.
		From \cref{thm:invariant_set}, it
		follows that the orbit of $h$ under $\Gamma(X)$ can only contain
		elements of $C_n$ and hence is finite. By~\cref{VGasStabiliser}, this is equivalent to the index
		$[\Gamma(X) : \Gamma(Y_h)]$ being finite.
	\end{proof}

While it is not true that every element of $C_n$ is a fixed point of $H^n$ for a given $n$, from
\cref{thm:characterization_finite_index} we can deduce the following weaker
statement.

\begin{cor}[Elements of $C_n$ are fixed points of some power of $H$]
	\label{cor:fixed_points_H_n_vs_elements_of_C_n}
	Let $G$ be a finite abelian group. Then $\bigcup\limits_{n\in \mathbb{N}} C_n =
	\bigcup\limits_{n \in \mathbb{N}} \{h \in \TheSequences_G : H^n h = h\}$.
\end{cor}

Given an element $h \in \TheSequences_G$, this corollary allows us to check whether the index of ~$\Gamma(Y_h)$ is finite in $\Gamma(X)$ without having to compute the orbit of $h$ under~$\langle H \rangle \leq \Gamma(X)$.

\begin{prop}[Algorithmic version of \cref{thm:characterization_finite_index}]
	Let $G$ be a finite abelian group of degree $d$ and $h \in \TheSequences_G$. 
	If $h$ is not periodic backwards and/or forwards,~$\Gamma(Y_h)$ has infinite index in $\Gamma(X)$.
	
	If $h$ is periodic backwards and forwards, choose $m \in \mathbb{N}$ to be the smallest natural number such that $h$ is $m$--periodic backwards and forwards. Furthermore, let $n \in \mathbb{N}$ such that $dn$ is the least common multiple of $m$ and $d$. Then $\Gamma(Y_h)$ has finite index in $\Gamma(X)$ if and only if $h$ fulfills the forper--backper relations at all places $0, \ldots, dn$ (with respect to~$dn$).

\begin{proof}
	If $[\Gamma(X) : \Gamma(Y_h)]$ is finite, by \cref{thm:characterization_finite_index,prop:fixed_points_H_n_in_C_n}, there has to exist an~$n \in \mathbb{N}$ such that $h$ is $dn$--periodic backwards and forwards. This proves the first statement.
	
	Now let $h$, $m$, and  $n$ be as described for the second statement.
	By \cref{thm:characterization_finite_index,cor:fixed_points_H_n_vs_elements_of_C_n}, we have that $[\Gamma(X) : \Gamma(Y_h)]$ is finite if and only if $h \in C_{n'}$ for some $n' \in \mathbb{N}$. Now we show that this is true if and only if~$h$ fulfills the forper--backper relations at all places $0, \ldots, dn$ (with respect to $dn$).

	``$\Rightarrow$'': Let $h \in C_{n'}$. In particular, $h$ is $dn'$--periodic backwards and forwards. By the choice of~$m$ and $n$, we have that $n'$ is a multiple of $n$. Together with $dn$--periodicity backwards and forwards, fulfilling the forper--backper relations with respect to $dn'$ then implies fulfilling them with respect to $dn$.

	``$\Leftarrow$'': Assume that $h$ fulfills the forper--backper relations with respect to~$dn$.
	As $h$ is $m$--periodic backwards and forwards and $dn$ is a multiple of $m$, $h$ is also $dn$--periodic as well as $d^2 n$--periodic backwards and forwards.
	Hence, the period sum of $h$ with respect to $d^2n$ is $d$ times the period sum with respect to $dn$ and therefore equal to $0$.

	As $h$ is $dn$--periodic, fulfilling the forper--backper relations with respect to $dn$ is equivalent to fulfilling the forper--backper relations with respect to $d^2 n$. Hence $h$ fulfills the forper--backper relations with respect to $d^2 n$ as well.
	Therefore, $h \in C_{d^2 n}$.
\end{proof}
\end{prop}

\section{Special case \texorpdfstring{$d=2$}{d=2}: Computation of the index}\label{sec:d=2}

In the case when the cover is of degree $d=2$, we can strengthen
\cref{thm:characterization_finite_index} to a version which enables an explicit
computation of the possible indices of the Veech group.

With $G = \mathbb{Z}/2\mathbb{Z}$, we immediately have from \cref{DescriptionBG} that $\TheSequences_G =
(\mathbb{Z}/2\mathbb{Z})^{\mathbb{Z}}\backslash\{\underline{0}\}$ is isomorphic
to $\Cov_G$ as the automorphism group of $G$ is trivial.
For convenience, we restate the results from \cref{ComputeTheGenerators,prop:action_H_n} for the case $G = \mathbb{Z}/2\mathbb{Z}$.

\begin{prop}[Action of $P_1$, $P_2$, $-\Id$, and $H^n$ for $d=2$]\label{prop:action_Zmod2Z}
	Let \(p \colon Y \to X\) be a covering of \(X\) with deck group \(G = \mathbb{Z}/2\mathbb{Z}\) and \(h\) be a defining vector of~\(\p\).
	The elements \(P_1\), \(P_2\), \(-\Id\), and $H^n$ for $n\geq 1$ act on~$[h]$ as follows:
	\begin{align*}
		(P_1\cdot h)_k &= h_{-k}, \\
		(P_2\cdot h)_k &=	h_{-1} + h_{-k-1}, \\
		(-\Id\cdot h)_k &= h_k,\\
		(H^n \cdot h)_k &= h_{k-n} + h_{-n}.
	\end{align*}
\end{prop}

In the following, we mimic the arguments of
\cref{sec:characterization_finite_index} but replace the set $C_n$ with a
smaller set $W_n$ with similar properties.

\begin{definition}[Weakly $n$--periodic vectors]
	Let $n\in \mathbb{N}$.
	A vector $h\in \TheSequences_{\mathbb{Z}/2\mathbb{Z}}$ is called \emph{weakly $n$--periodic} if $h_{k+n} - h_k = h_n$ holds for every~$k \in \mathbb{Z}$.
	
	We denote the set of weakly $n$--periodic vectors in $\TheSequences_{\mathbb{Z}/2\mathbb{Z}}$ by $W_n$.
\end{definition}

Note that we work now in $G = \mathbb{Z}/2\mathbb{Z}$ where $x = -x$ for every $x \in G$. Hence setting~$k= -n$ in the definition yields $h_n = h_0 - h_{-n} = h_{-n}$ for every $h\in W_n$.

The set $W_n$ coincides with the set of fixed points of $H^n$, as we show now.

\begin{prop}[Fixed points of $H^n$] \label{lem:fixed_points_H_n_degree_2}
 Let $h \in \TheSequences_{\mathbb{Z}/2\mathbb{Z}}$ and $n\in \mathbb{Z}$.
 Then $h$ is a fixed point of $H^n$ if and only if~$h$ is weakly $n$--periodic.
\end{prop}

\begin{proof}
 	If $h$ is a fixed point of $H^n$, then we obtain from \cref{prop:action_Zmod2Z}
 	\begin{equation*}
 		h_{k+n} - h_k = (H^n \cdot h)_{k+n} - h_k = h_{(k+n)-n} + h_{-n} - h_k = h_n
 	\end{equation*}
 	for every~$k \in \mathbb{Z}$ which is the definition of weak $n$-periodicity.
 
 	If $h$ is weakly $n$--periodic, we obtain 
 	\begin{equation*}
 		(H^n \cdot h)_{k}
 		= h_{k-n} + h_{-n}
 		= h_{k-n} + h_n
 		= h_{(k-n) +n}
 		= h_k
 	\end{equation*}
 	for every $k\in \mathbb{N}$ which shows that $h$ is a fixed point of $H^n$.
\end{proof}

Note that weakly $n$--periodic vectors in $\TheSequences_{\mathbb{Z}/2\mathbb{Z}}$ are always $2n$--periodic and are contained in~$C_n$.
In particular, the set~$W_n$ of weakly $n$--periodic vectors is finite for every $n\in \mathbb{N}$.

Moreover, for each $h\in W_n$, there exists a smallest $m \in \mathbb{N}$ such
that $h$ is a fixed point of~$H^m$. Even more, we have that this $m$ is a
divisor of $n$. On the contrary, $W_m \subseteq W_n$ for every divisor $m \in \mathbb{N}$ of~$n$, hence we can define $W_n^\ast \coloneqq W_n \setminus
\bigcup\limits_{m \mid n, m\neq n} W_m$ as the set of $h \in
\TheSequences_{\mathbb{Z}/2\mathbb{Z}}$ such that $h$ is a fixed point of~$H^n$
but not of any $H^m$ for $m \in \mathbb{N}$, $m<n$.

In the next lemma, we see that $W_n$ is a suitable replacement of $C_n$ in the sense that it
is an invariant set under the action of~$\Gamma(X)$. As $\Gamma(X)$ acts by bijections on $W_n$, it follows that also~$W_n^\ast$ is an invariant set under~$\Gamma(X)$.

\begin{lem}[$\Gamma(X)$ preserves weak periodicity] \label{lem:Veech_group_preserves_weak_periodicity}
	For any $n\in\mathbb{N}$, we have that $\Gamma(X) \cdot W_n \subseteq W_n$.
\end{lem}

\begin{proof}
	Let $h \in \TheSequences_{\mathbb{Z}/2\mathbb{Z}}$ be a weakly $n$--periodic
	vector. We check through calculations in~$\mathbb{Z}/2\mathbb{Z}$ that~$-\Id \cdot h$, $P_1 \cdot h$, and $P_2 \cdot h$ are weakly $n$--periodic as well. 
	For $k \in \mathbb{Z}$, we have, by \cref{prop:action_Zmod2Z}, that:
	\begin{align*}
		(-h)_{k+n} - (-h)_k
		& = -h_{k+n} + h_k = h_n = (-h)_n, \\
		(P_1 h)_{k+n} - (P_1 h)_k
		& = h_{-k-n} - h_{-k} = h_n = h_{-n} = (P_1 h)_n.
	\end{align*}
	To prove the statement for $P_2$,
	we use the weak $n$--periodicity of $h$ first at place $-k-n-1$ and then at place $-n-1$:
	\begin{equation*}
		(P_2 h)_{k+n} - (P_2 h)_k = h_{-1} + h_{-k-n-1} - (h_{-1} + h_{-k-1})
		= - h_n 
		= h_{-n-1} + h_{-1} = (P_2 h)_n.
		\qedhere
	\end{equation*}
\end{proof}

With these information on the set $W_n$, we can now proceed as in \cref{sec:characterization_finite_index} and prove \cref{thm:criterion_finite_index_p=2} which states that the Veech group $\Gamma(Y_h)$ of a cover $Y_h$ has finite index in $\Gamma(X)$ if and only if the defining vector $h$ of $Y_h$ is periodic.

\begin{proof}[Proof of \cref{thm:criterion_finite_index_p=2}]
	Analogously to the proof of \cref{thm:characterization_finite_index}, we have that the Veech group of $Y_h$ has finite index in the Veech group of the Chamanara surface if and only if $h$ is weakly $n$--periodic for some $n \in \mathbb{N}$ (which is by \cref{lem:fixed_points_H_n_degree_2} equivalent to being a fixed point of $H^n$).
	
	As $h_0 = 0$ for every $h \in \TheSequences_{\mathbb{Z}/2\mathbb{Z}}$, every $n$--periodic vector is also weakly $n$--periodic. Furthermore, every weakly $n$--periodic vector is also $2n$--periodic. Hence, the elements of $\cup_{n \in \mathbb{N}} W_n$ are exactly the periodic vectors.
\end{proof}

To prove \cref{thm:realization_F_n}, we now study the Veech groups of finite index in more detail by describing their Schreier coset graphs in \cref{prop:possible_Schreier_coset_graphs,prop:all_Schreier_coset_graphs} and calculating their index, see \cref{cor:every_free_group_is_projective_Veech_group}.

For a given subgroup $H$ of a group~$G$ with generating set $S$, the vertices of the \emph{Schreier coset graph} are the cosets $gH = \{gh \mid h \in H\}$ for~$g \in G$. There is an edge between two cosets~$gH$ and~$g'H$ if there exists an $s \in S$ such that $sgH = g'H$, similar to Cayley graphs.
In particular, the number of vertices in the Schreier coset graph of $H$ is the index of $H$ in $G$.

In the following, we always suppress the edges coming from $-\Id$ in the Schreier coset graphs as $-\Id$ is contained in all of the groups involved and all edges labelled with $-\Id$ would be loops.

\begin{prop}[Possible Schreier coset graphs for covers from elements in $W_n$]
	\label{prop:possible_Schreier_coset_graphs}
	Let $n \in \mathbb{N}$ and $h \in W_n^\ast$. For the
	subgroup $\Gamma(Y_h)$ in $\Gamma(X)$, there are two possible types of
	Schreier coset graphs, which we call \emph{Striezel type} and \emph{Kranz
	type} (see \cref{fig:kranz_striezel} on Page~\pageref{fig:kranz_striezel} for the names).

	A graph of Striezel type has $n$ vertices and is of this form:

\begin{center}
  \begin{tikzpicture}[scale=2]
   \draw[fill] (0,0) circle (1pt);
   \draw[fill] (1,0) circle (1pt);
   \draw[fill] (2,0) circle (1pt);
   \draw[fill] (3,0) circle (1pt);
   \draw[fill] (4,0) circle (1pt);
   \draw[fill] (5,0) circle (1pt);
   \draw[fill] (6,0) circle (1pt);

   \draw[->] (-0.1,0.1) arc (20:340:0.3) node[left]{$P_2$};

   \draw[->] (0.1,0.1) to [bend left] node[above]{$P_1$} (0.9,0.1);
   \draw[->] (1.1,0.1) to [bend left] node[above]{$P_2$} (1.9,0.1);
   \draw[->] (2.1,0.1) to [bend left] node[above]{$P_1$} (2.9,0.1);
   \draw (3.5,0) node {$\ldots$};
   \draw[->] (4.1,0.1) to [bend left] node[above]{$P_1$} (4.9,0.1);
   \draw[->] (5.1,0.1) to [bend left] node[above]{$P_2$} (5.9,0.1);

   \draw[->] (6.1,0.1) arc (160:-160:0.3) node[right]{$P_1$};

   \draw[->] (5.9,-0.1) to [bend left] node[below]{$P_2$} (5.1,-0.1);
   \draw[->] (4.9,-0.1) to [bend left] node[below]{$P_1$} (4.1,-0.1);
   \draw[->] (2.9,-0.1) to [bend left] node[below]{$P_1$} (2.1,-0.1);
   \draw[->] (1.9,-0.1) to [bend left] node[below]{$P_2$} (1.1,-0.1);
   \draw[->] (0.9,-0.1) to [bend left] node[below]{$P_1$} (0.1,-0.1);
  \end{tikzpicture}
\end{center}

  A graph of Kranz type has $2n$ vertices and is of this form:

\begin{center}
  \begin{tikzpicture}[scale=2]
   \draw[fill] (22.5:1) circle (1pt);
   \draw[fill] (67.5:1) circle (1pt);
   \draw[fill] (112.5:1) circle (1pt);
   \draw[fill] (157.5:1) circle (1pt);
   \draw[fill] (202.5:1) circle (1pt);
   \draw[fill] (247.5:1) circle (1pt);
   \draw[fill] (292.5:1) circle (1pt);
   \draw[fill] (337.5:1) circle (1pt);

   \draw (0,-1) node {$\ldots$};

   \draw[->] (240:1.1) to [bend left] node[left]{$P_2$} (210:1.1);
   \draw[->] (195:1.1) to [bend left] node[left]{$P_1$} (165:1.1);
   \draw[->] (150:1.1) to [bend left] node[left]{$P_2$} (120:1.1);
   \draw[->] (105:1.1) to [bend left] node[above]{$P_1$} (75:1.1);
   \draw[->] (60:1.1) to [bend left] node[above]{$P_2$} (30:1.1);
   \draw[->] (15:1.1) to [bend left] node[right]{$P_1$} (-15:1.1);
   \draw[->] (-30:1.1) to [bend left] node[right]{$P_2$} (-60:1.1);

   \draw[->] (-60:0.9) to [bend left] node[above]{$P_2$} (-30:0.9);
   \draw[->] (-15:0.9) to [bend left] node[left]{$P_1$} (15:0.9);
   \draw[->] (30:0.9) to [bend left] node[below]{$P_2$} (60:0.9);
   \draw[->] (75:0.9) to [bend left] node[below]{$P_1$} (105:0.9);
   \draw[->] (120:0.9) to [bend left] node[below]{$P_2$} (150:0.9);
   \draw[->] (165:0.9) to [bend left] node[right]{$P_1$} (195:0.9);
   \draw[->] (210:0.9) to [bend left] node[above]{$P_2$} (240:0.9);
  \end{tikzpicture}
\end{center}
\end{prop}

\begin{proof}
	Recall that the elements of $\Gamma(Y_h)$ are exactly the elements of $\Gamma(X)$ which lift to~$Y_h$. Hence \cref{prop:powers_parabolic_lift} states	that $P_1^2$ and $P_2^2$	lift to $\Gamma(Y_h)$. This implies that
	$P \cdot h$ is a fixed point of~$P_1^2$ and $P_2^2$ for any $P \in \Gamma(X)$.
	In particular, any conjugates of $P_1^2$ and~$P_2^2$ are contained in
	$\Gamma(Y_h)$. Furthermore, $-\Id \in \Gamma(Y_h)$ by \cref{rem:Id_lifts}.
	Hence, the normal subgroup~$N \coloneqq \llangle P_1^2, P_2^2, -\Id
	\rrangle$ is contained in $\Gamma(Y_h)$.

	Because $N$ is normal, its Schreier coset graph is the Cayley graph of
	$\Gamma(X)/N$ and hence has as vertices all elements of the form $P_1^{\pm
	1} P_2^{\pm 1} P_1^{\pm 1} \cdots$ or $P_2^{\pm 1} P_1^{\pm 1} P_2^{\pm 1}
	\cdots$. Therefore, the Schreier coset graph of $N$ is the following:
\begin{center}
	\begin{tikzpicture}[scale=2]
		\clip (-0.6,-0.45) rectangle (6.6,0.45);
		\draw[->, densely dotted] (-0.9,0.1) to [bend left] (-0.1,0.1);
		\draw[->, densely dotted] (6.1,0.1) to [bend left] (6.9,0.1);
		\draw[->, densely dotted] (6.9,-0.1) to [bend left] (6.1,-0.1);
		\draw[->, densely dotted] (-0.1,-0.1) to [bend left] (-0.9,-0.1);

		\clip (-0.3,-0.45) rectangle (6.3,0.45);
		\draw[fill] (0,0) circle (1pt);
		\draw[fill] (1,0) circle (1pt);
		\draw[fill] (2,0) circle (1pt);
		\draw[fill] (3,0) circle (1pt);
		\draw[fill] (4,0) circle (1pt);
		\draw[fill] (5,0) circle (1pt);
		\draw[fill] (6,0) circle (1pt);

		\draw[->] (0.1,0.1) to [bend left] node[above]{$P_1$} (0.9,0.1);
		\draw[->] (1.1,0.1) to [bend left] node[above]{$P_2$} (1.9,0.1);
		\draw[->] (2.1,0.1) to [bend left] node[above]{$P_1$} (2.9,0.1);
		\draw[->] (3.1,0.1) to [bend left] node[above]{$P_2$} (3.9,0.1);
		\draw[->] (4.1,0.1) to [bend left] node[above]{$P_1$} (4.9,0.1);
		\draw[->] (5.1,0.1) to [bend left] node[above]{$P_2$} (5.9,0.1);

		\draw[->] (5.9,-0.1) to [bend left] node[below]{$P_2$} (5.1,-0.1);
		\draw[->] (4.9,-0.1) to [bend left] node[below]{$P_1$} (4.1,-0.1);
		\draw[->] (3.9,-0.1) to [bend left] node[below]{$P_2$} (3.1,-0.1);
		\draw[->] (2.9,-0.1) to [bend left] node[below]{$P_1$} (2.1,-0.1);
		\draw[->] (1.9,-0.1) to [bend left] node[below]{$P_2$} (1.1,-0.1);
		\draw[->] (0.9,-0.1) to [bend left] node[below]{$P_1$} (0.1,-0.1);
	\end{tikzpicture}
\end{center}

As $N \leq \Gamma(Y_h)$, the Schreier coset graph of $N$ is a cover of the
Schreier coset graph of~$\Gamma(Y_h)$, hence it can only be of Striezel type or of Kranz type.
Recall that by definition $H^n$ is contained in $\Gamma(Y_h)$ but not $H^m$ for any $m < n$,  and hence $H^n$ but not $H^m$ for $m < n$  acts trivially on the Schreier coset graph of $\Gamma(Y_h)$. As $H = -P_1 P_2$, the structure of the graphs shows that a graph of Striezel type has $n$ vertices and a graph of Kranz type has $2n$ vertices.
\end{proof}

Note that for $n$ odd, a graph of Striezel type contains one fixed point of
$P_1$ and one of $P_2$ whereas for $n$ even, it contains either two fixed points
of $P_1$ or two fixed points of $P_2$ (and no fixed point of the other parabolic
generator).

\begin{exa}[The set $W_n^\ast$ for small $n$] \label{exa:W_n_ast}
	\begin{enumerate}
		\item The set $W_1 = W_1^\ast$ contains exactly one element. That means
		that (except for the disconnected cover) there is only one cover of
		degree $2$ which has Veech group $\Gamma(X)$. It corresponds to $h =
		(\ldots, 1, 0, 1, h_0 = 0, 1,0,1,0, \ldots)$.

		Hence, the graph that makes up $W_1^\ast$ is
			\raisebox{-1.3em}{\begin{tikzpicture}[scale=1.8]
				\draw[fill] (0,0) circle (1pt) node[below=2pt]{$h$};

				\draw[->] (-0.1,-0.1) arc (340:20:0.3) node[left]{$P_2$};
				\draw[->] (0.1,-0.1) arc (-160:160:0.3) node[right]{$P_1$};
			\end{tikzpicture}}
		.
\pagebreak

		\item The set $W_2^\ast$ consists of exactly one orbit which is
		\begin{center}
			\begin{tikzpicture}[scale=1.8]
				\draw[fill] (0,0) circle (1pt) node[below=7pt]{$h^{0,1}$};
				\draw[fill] (1,0) circle (1pt) node[below=7pt]{$h^{1,1}$};

				\draw[->] (-0.1,-0.1) arc (340:20:0.3) node[left]{$P_2$};

				\draw[->] (0.1,0.1) to [bend left] node[above]{$P_1$} (0.9,0.1);

				\draw[->] (1.1,-0.1) arc (-160:160:0.3) node[right]{$P_2$};

				\draw[->] (0.9,-0.1) to [bend left] node[above]{$P_1$} (0.1,-0.1);
			\end{tikzpicture}
		\end{center}
		where $h^{0,1}$ corresponds to the bi-infinite vector $(\ldots, 0, 0, 1, 1, h^{0,1}_0 = 0, 0,
		1, 1, 0, \ldots)$ and~$h^{1,1}$ corresponds to $(\ldots, 0, 1, 1, 0,
		h^{1,1}_0 = 0, 1, 1, 0, 0, \ldots)$.

		\item The set $W_3^\ast$ consists of two orbits which are
		\begin{center}
			\begin{tikzpicture}[scale=1.8]
				\draw[fill] (0,0) circle (1pt) node[below=9pt]{$h^{0,1,1}$};
				\draw[fill] (1,0) circle (1pt) node[below=9pt]{$h^{1,1,1}$};
				\draw[fill] (2,0) circle (1pt) node[below=9pt]{$h^{0,0,1}$};

				\draw[->] (-0.1,-0.1) arc (340:20:0.3) node[left]{$P_1$};

				\draw[->] (0.1,0.1) to [bend left] node[above]{$P_2$} (0.9,0.1);
				\draw[->] (1.1,0.1) to [bend left] node[above]{$P_1$} (1.9,0.1);

				\draw[->] (2.1,-0.1) arc (-160:160:0.3) node[right]{$P_2$};

				\draw[->] (1.9,-0.1) to [bend left] node[above]{$P_1$} (1.1,-0.1);
				\draw[->] (0.9,-0.1) to [bend left] node[above]{$P_2$} (0.1,-0.1);

				\begin{scope}[xshift=3.8cm]
					\draw[fill] (0,0) circle (1pt) node[below=9pt]{$h^{1,1,0}$};
					\draw[fill] (1,0) circle (1pt) node[below=9pt]{$h^{0,1,0}$};
					\draw[fill] (2,0) circle (1pt) node[below=9pt]{$h^{1,0,0}$};

					\draw[->] (-0.1,-0.1) arc (340:20:0.3) node[left]{$P_1$};

					\draw[->] (0.1,0.1) to [bend left] node[above]{$P_2$} (0.9,0.1);
					\draw[->] (1.1,0.1) to [bend left] node[above]{$P_1$} (1.9,0.1);

					\draw[->] (2.1,-0.1) arc (-160:160:0.3) node[right]{$P_2$};

					\draw[->] (1.9,-0.1) to [bend left] node[above]{$P_1$} (1.1,-0.1);
					\draw[->] (0.9,-0.1) to [bend left] node[above]{$P_2$} (0.1,-0.1);
				\end{scope}
			\end{tikzpicture}
		\end{center}
		where the superscript indicates the entries at the first, second, and third place.

		\item The set $W_4^\ast$ consists of three orbits which are
		\begin{center}
			\begin{tikzpicture}[scale=1.8]
				\draw[fill] (0,0) circle (1pt) node[below=12pt]{$h^{0,1,0,0}$};
				\draw[fill] (1,0) circle (1pt) node[below=12pt]{$h^{1,0,0,0}$};
				\draw[fill] (2,0) circle (1pt) node[below=12pt]{$h^{0,0,1,0}$};
				\draw[fill] (3,0) circle (1pt) node[below=12pt]{$h^{1,1,1,0}$};

				\draw[->] (-0.1,-0.1) arc (340:20:0.3) node[left]{$P_1$};

				\draw[->] (0.1,0.1) to [bend left] node[above]{$P_2$} (0.9,0.1);
				\draw[->] (1.1,0.1) to [bend left] node[above]{$P_1$} (1.9,0.1);
				\draw[->] (2.1,0.1) to [bend left] node[above]{$P_2$} (2.9,0.1);

				\draw[->] (3.1,-0.1) arc (-160:160:0.3) node[right]{$P_1$};

				\draw[->] (2.9,-0.1) to [bend left] node[above]{$P_2$} (2.1,-0.1);
				\draw[->] (1.9,-0.1) to [bend left] node[above]{$P_1$} (1.1,-0.1);
				\draw[->] (0.9,-0.1) to [bend left] node[above]{$P_2$} (0.1,-0.1);

				\begin{scope}[yshift=-1.2cm]
					\draw[fill] (0,0) circle (1pt) node[below=12pt]{$h^{0,0,0,1}$};
					\draw[fill] (1,0) circle (1pt) node[below=12pt]{$h^{1,1,1,1}$};
					\draw[fill] (2,0) circle (1pt) node[below=12pt]{$h^{0,0,1,1}$};
					\draw[fill] (3,0) circle (1pt) node[below=12pt]{$h^{0,1,1,1}$};

					\draw[->] (-0.1,-0.1) arc (340:20:0.3) node[left]{$P_2$};

					\draw[->] (0.1,0.1) to [bend left] node[above]{$P_1$} (0.9,0.1);
					\draw[->] (1.1,0.1) to [bend left] node[above]{$P_2$} (1.9,0.1);
					\draw[->] (2.1,0.1) to [bend left] node[above]{$P_1$} (2.9,0.1);

					\draw[->] (3.1,-0.1) arc (-160:160:0.3) node[right]{$P_2$};

					\draw[->] (2.9,-0.1) to [bend left] node[above]{$P_1$} (2.1,-0.1);
					\draw[->] (1.9,-0.1) to [bend left] node[above]{$P_2$} (1.1,-0.1);
					\draw[->] (0.9,-0.1) to [bend left] node[above]{$P_1$} (0.1,-0.1);
				\end{scope}

				\begin{scope}[yshift=-2.4cm]
					\draw[fill] (0,0) circle (1pt) node[below=12pt]{$h^{1,0,1,1}$};
					\draw[fill] (1,0) circle (1pt) node[below=12pt]{$h^{0,1,0,1}$};
					\draw[fill] (2,0) circle (1pt) node[below=12pt]{$h^{1,0,0,1}$};
					\draw[fill] (3,0) circle (1pt) node[below=12pt]{$h^{1,1,0,1}$};

					\draw[->] (-0.1,-0.1) arc (340:20:0.3) node[left]{$P_2$};

					\draw[->] (0.1,0.1) to [bend left] node[above]{$P_1$} (0.9,0.1);
					\draw[->] (1.1,0.1) to [bend left] node[above]{$P_2$} (1.9,0.1);
					\draw[->] (2.1,0.1) to [bend left] node[above]{$P_1$} (2.9,0.1);

					\draw[->] (3.1,-0.1) arc (-160:160:0.3) node[right]{$P_2$};

					\draw[->] (2.9,-0.1) to [bend left] node[above]{$P_1$} (2.1,-0.1);
					\draw[->] (1.9,-0.1) to [bend left] node[above]{$P_2$} (1.1,-0.1);
					\draw[->] (0.9,-0.1) to [bend left] node[above]{$P_1$} (0.1,-0.1);
				\end{scope}
			\end{tikzpicture}
		\end{center}
		where the superscript indicates the entries at the first, second, third, and fourth place.
	\end{enumerate}
\end{exa}

Let us determine the cardinality of $W_n^\ast$ for all $n\in
\mathbb{N}$. Observe first that an element of $W_n$ is fully described by the
entries $h_1, \ldots, h_n$. Furthermore, each such description defines an
element of~$W_n \cup \{ \underline{0} \}$. Hence, we have
\begin{equation*}
	|W_n^\ast| = 2^n - \sum\limits_{m\mid n, m \neq n} |W_m^\ast|  - 1 = \sum\limits_{m\mid n} \mu(m) \cdot 2^{\sfrac{n}{m}}
\end{equation*}
where $\mu$ is the Möbius function, see \cite[Theorem 1.5.5]{allouche_shallit_03}.

We will show now that graphs of Striezel type appear for any $n\in \mathbb{N}$.
For this, we first count the fixed points of $P_1$ and $P_2$ in $W_n$.

\begin{lem}[Number of graphs of Striezel type]\label{lem:number_Striezels} Let
	$n\in\mathbb{N}$. Then, in the set $W_n$, the number of fixed
	points~of
	\begin{enumerate}
		\item $P_1$ is $2^{\lfloor \sfrac{(n+1)}{2} \rfloor}-1$,
		\item $P_2$ is $2^{\lfloor \sfrac{n}{2} \rfloor +1}-1$,
		\item $P_1$ and $P_2$ is $1$.
	\end{enumerate}

	Therefore, the number of graphs of Striezel type is $3 \cdot
	2^{\sfrac{n}{2}-1} - 1$ for $n$ even and  $2^{\sfrac{(n+1)}{2}} -1$ for $n$~odd.
\end{lem}

\begin{proof}
	Recall that each choice of entries $h_1, \ldots, h_n$, which are not all zeros, defines an element of $W_n$.

	\begin{enumerate}
	\item For a fixed point $h$ of $P_1$, we have for every $k \in \mathbb{Z}$,
	that $h_k = (P_1 h)_k = h_{-k}$. Therefore, we have the additional
	restrictions $h_k = h_{-k} = h_{n-k}+h_n$.

	For $n$ even, we obtain $h_n = h_{\sfrac{n}{2}} - h_{n-\sfrac{n}{2}} = 0$.
	This condition, together with the elements $h_1, \ldots, h_{\sfrac{n}{2}}$,
	describes $h$ completely. As not all of them can be $0$, we have $2^{\sfrac{n}{2}}-1$ choices for
	a fixed point.

	For $n$ odd, the elements $h_1, \ldots, h_{\sfrac{(n-1)}{2}}$, and $h_n$ describe $h$
	completely. Therefore, we have $2^{\sfrac{(n+1)}{2}} -1$ choices for a fixed
	point.

	\item For a fixed point $h$ of $P_2$, for every $k \in \mathbb{Z}  \setminus
	\{-1\}$, we get $h_k = (P_2(h))_k = h_{-1} + h_{-k-1} = h_{n-1} +
	h_{n-k-1}$.

	That is, for $n$ even, the elements $h_{\sfrac{n}{2}}, \ldots, h_n$ describe
	$h$ completely. Therefore, we have $2^{\sfrac{n}{2} + 1} -1$ choices for a
	fixed point.

	For $n$ odd, $n \geq 3$, we obtain with $k= \frac{n-1}{2}$ that $h_{n-1} =
	0$ (and, therefore, $h_k = h_{n-k-1}$). This condition, together with the
	elements $h_{\sfrac{(n-1)}{2}}, \ldots, h_{n-2}, h_n$, describes $h$
	completely and we have $2^{\sfrac{(n+1)}{2}} -1$ choices for a fixed point.

	\item If $P_1$ and $P_2$ both are fixing $h$, a combination of the arguments from (i) and (ii) shows $h_k = h_{-1} + h_{-k-1} = h_1 + h_{k+1}$ for all $k \in \mathbb{Z} \setminus \{-1\}$. As $h_1 = 0$ would imply $h = 0$ (which corresponds to the excluded disconnected cover), we have $h_1 = 1$ and hence $h= (\ldots, 0, 1, h_0 = 0, 1, 0, \ldots)$.
	\end{enumerate}

	The Schreier coset graph of $\Gamma(Y_h)$ is of Striezel type if and only if it contains exactly two fixed points of $P_1$ or $P_2$ or consists of one vertex (which is a fixed point of both). Hence the number of graphs of Striezel type is $(2^{\lfloor \sfrac{(n+1)}{2} \rfloor} +
	2^{\lfloor \sfrac{n}{2} \rfloor +1} -2) /2$ which is $3 \cdot
	2^{\sfrac{n}{2}-1} -1$ for $n$ even and  $2^{\sfrac{(n+1)}{2}} -1$ for $n$ odd.
\end{proof}

If $n$ is prime and $n\neq 2$, we can give the number of graphs of Striezel and
Kranz type in $W_n^\ast$ in closed form. The number of graphs of Striezel type in $W_n$ is $2^{\sfrac{(n+1)}{2}} -1$ of which one is contained in~$W_1$.
Hence, the number of graphs of Striezel
type in $W_n^\ast$ is $2^{\sfrac{(n+1)}{2}} - 2$. As graphs of Striezel type
have $n$ elements, $n \cdot (2^{\sfrac{(n+1)}{2}} - 2)$ of the~$2^n -2$ elements
of $W_n^\ast$ are contained in graphs of this type. As graphs of Kranz type
have $2n$ elements, we can calculate the number of graphs of Kranz type~as
\begin{equation*}
	\frac{(2^n - 2) - n\cdot (2^{\sfrac{(n+1)}{2}} - 2)}{2n}
	= \frac{2^{n-1} - 1}{n} - 2^{\sfrac{(n-1)}{2}} +1
	.
\end{equation*}

For general $n$, \cref{lem:number_Striezels} yields a lower bound of $2^{\sfrac{(n+1)}{2}} -1$ on the number of graphs of Striezel type in $W_n$.
To obtain a (very rough) lower bound on the number of graphs of Striezel type in~$W_n^\ast$, we subtract $|W_{\lfloor \sfrac{n}{2} \rfloor}|$, a bound on the number of vectors that could potentially be in graphs of Striezel type in $W_n \setminus W_n^\ast$. This yields $2^{\sfrac{(n+1)}{2}} -1 - (2^{\lfloor \sfrac{n}{2} \rfloor} -1) > 0$, hence graphs of Striezel type exist for every $n\in \mathbb{N}$.

For the number of graphs of Kranz type in $W_n^\ast$, note that
\begin{equation*}
	|W_n^\ast| \geq 2^n - \sum\limits_{m\mid n, m \neq n} 2^m \geq 2^n - \sum\limits_{m \leq \sfrac{n}{2}} 2^m \geq 2^n - 2^{\sfrac{n}{2} +1} % \geq 2^{n-1}
	.
\end{equation*}
Therefore, with the upper bound of $3 \cdot 2^{\sfrac{n}{2} -1}$ for graphs of Striezel type in $W_n^\ast$ from \cref{lem:number_Striezels},
we obtain a lower bound for the number of graphs of Kranz type in $W_n^\ast$ by
\begin{equation*}
	\frac{|W_n^\ast| - n \cdot 3 \cdot 2^{\sfrac{n}{2} -1} }{2n} \geq \frac{2^n - (3n+4) 2^{\sfrac{n}{2}-1}}{2n},
\end{equation*}
which is positive for $n\geq 8$. Checking the small cases separately as in
\cref{exa:W_n_ast}, we see that we have graphs of Kranz type if and only if $n \geq 6$.

\medskip

For the next proposition, we recall that the graphs of Striezel and of Kranz
type have by definition a loop at any vertex labelled by $-\Id$ which is omitted
in the graphs in \cref{prop:possible_Schreier_coset_graphs}.

\begin{prop}[All possible Schreier coset graphs] \label{prop:all_Schreier_coset_graphs}
  A graph with \(n\) vertices is the Schreier coset graph of the Veech group of
  a normal covering \(p\colon Y_h \to X\) of degree~$2$ with respect to the generating
  system \(P_1, P_2,-\Id\) if and only if it is of Striezel type or of Kranz
  type and it is neither of Striezel type with two vertices whose two loops are
  labelled by~\(P_1\) nor of Kranz type with less than six vertices.
\end{prop}

\begin{proof}
	For \(n \in \{1,2,3,4,5\}\), we obtain from \cref{lem:number_Striezels} the
	following table for the number~$A_1(n)$ and~\(A_2(n)\) of fixed points of
	\(P_1\) and \(P_2\) in \(W_n\) and the  number \(B_1(n)\) and \(B_2(n)\) of
	fixed points of \(P_1\) and \(P_2\) in \(W_n^\ast\):
	\[
	\begin{array}{lccccc}
		\toprule
		n      &1 &2 &3 &4 &5\\
		\midrule
		A_1(n) & 1&1 &3 &3 & 7\\
		B_1(n) & 1&0 &2 &2 & 6\\
		|W_n^\ast| & 1 & 2 & 6 & 12 & 30\\
		\bottomrule
	\end{array} \hspace*{10mm}
		\begin{array}{lccccc}
		\toprule
		n      &1 &2 &3 &4 &5\\
		\midrule
		A_2(n) & 1&3 &3 &7 & 7\\
		B_2(n) & 1&2 &2 &4 & 6\\
		|W_n^\ast| & 1 & 2 & 6 & 12 & 30\\
		\bottomrule
	\end{array}
	\]
	It follows from the table that, for \(n=2\), we have no graph of Striezel type with
	\(P_1\)--loops and only one graph of Striezel type with \(P_2\)--loops which
	consists of two vertices.  Since \(|W_2^\ast| = 2\), this gives us the full
	set~\(W_2^\ast\). The statement follows similarly from the table for \(n \in
	\{3,4,5\}\). For \(n \geq 6\), we get from \cref{lem:number_Striezels} and the previous discussion that
	all three types of graphs occur.
\end{proof}

We conclude this section with the proof of \cref{thm:realization_F_n} which states that every free group appears as projective Veech group. 
For this, we reformulate the previous discussion as a statement about the rank of the projective Veech
groups that appear for covers. As the projective Veech group $P\Gamma(X)$ is a
free group, every subgroup $P\Gamma(Y)$ is also free by Schreier's theorem. The
rank of a subgroup of index $n$ in a free group of rank $k$ is $n(k-1)+1$.
Together with the existence of Schreier coset graphs of Striezel type in each
$W_n^\ast$, this shows immediately:

\begin{cor}[Every free group is a projective Veech group of a cover] \label{cor:every_free_group_is_projective_Veech_group}
	For every $n \geq 2$, there exists a cover $Y_h$ of $X$ such that the
	projective Veech group of $Y_h$ is a free group of rank $n$.
\end{cor}

Note that there exist also translation surfaces of finite area and infinite type whose Veech group is isomorphic to $\mathbb{Z}$, see for example the exponential surface in \cite[Proposition 2.4]{randecker_16}.
Together with this observation, \cref{cor:every_free_group_is_projective_Veech_group} proves \cref{thm:realization_F_n}.

\section{Devious finite abelian covers of the Chamanara surface}\label{sec:devious}

In this section, we prove the \hyperlink{cor:devious}{Corollary to Theorem 3}. We begin by recalling the relevant
terminology from~\cite{hooper_trevino_19}.

Two translation surfaces $X$ and $X'$ are \emph{translation equivalent} if there
exists a translation $h \colon X \to X'$. Let $\SL_\pm(2, \RR)$ be the subgroup of
$\GL(2,\RR)$ of matrices with determinant equal to~$\pm 1$, which acts by
post-composition on translation surfaces. Then, we define the
$\SL_\pm(2,\RR)$--orbit of the surface $X$ as
\[
\mathcal{O}(X) = \{A\cdot X: A\in\SL_\pm (2,\RR)\} / \text{translation equivalence}.
\]
Let $\Gamma_\pm (X)$ be the \emph{extended} Veech group of $X$, which is the
stabilizer of $X$ under the action of~$\SL_\pm(2,\RR)$. We remark that two
surfaces $A_1\cdot X$ and $A_2\cdot X$ in the orbit of $X$ are translation
equivalent if and only if $A_1$ and $A_2$ are in the same left coset of
$\Gamma_\pm (X)$. Hence, we can identify~$\mathcal{O}(X)$ with $\SL_\pm(2,\RR)/
\Gamma_\pm (X)$. With this identification, the orbit space $\mathcal{O}(X)$
inherits a topology from the usual topology of $\SL_\pm(2,\RR)$.

The \emph{Teichmüller orbit} of $X$ is the subset
\[
\{g_t \cdot X: t\in\RR\} / \text{translation equivalence}
\subseteq \mathcal{O}(X),
\] 
where $g_t = \left(\begin{smallmatrix}
	e^t & 0 \\
	0	& e^{-t} \end{smallmatrix}\right)$ for $t \in \mathbb{R}$. We say
that the Teichmüller orbit is \emph{non-divergent} if there is a sequence
$(t_n)$ with $t_n \to \infty$ as $n \to \infty$ such that $g_{t_n}\cdot X$
converges in~$\mathcal{O}(X)$ as $n\to\infty$.

Let $G$ be a finite group and define, similar to \cref{sec:monodromy}, the space
$\Cov^\textnormal{HT}_G(X)$ of (finite) covers with deck group $G$, up to
isomorphisms of the surface $X$. Note that we allow here non-normal and disconnected covers. Then, one can build the space formed by the
union of $\SL_\pm(2,\RR)$--orbits of covers of $X$ with deck group $G$:
\[
\widetilde{\mathcal{O}}_G(X) = \{A\cdot Y: A\in\SL_\pm (2,\RR), Y\in\Cov^\textnormal{HT}_G (X)\} / \text{translation equivalence}.
\]
We can extend the action of the group $\{g_t | t\in\RR \}$ to
$\widetilde{\mathcal{O}}_G(X)$ and endow $\widetilde{\mathcal{O}}_G(X)$ with the
quotient topology, induced by the one of $\SL_\pm(2,\RR)\times\Cov^\textnormal{HT}_G(X)$
where $\Cov^\textnormal{HT}_G(X)$ is seen as a subquotient of $\Hom(\pi_1(X),G)$.

Generically, connected covers accumulate under the Teichmüller action on
$\widetilde{\mathcal{O}}_G(X)$ to connected covers. More precisely:

\begin{prop}[Accumulation of Teichmüller orbit {\cite[Theorem 5]{hooper_trevino_19}}] Let $X$ be an infinite
	translation surface of finite area and assume that its
	Teichmüller orbit is non-divergent in $\mathcal{O}(X)$. Given an integer
	$n\ge 2$, consider a subgroup $G$ of $S_n$, which acts transitively on
	$\{1,\dotsc, n\}$. Then, almost every element of $\Cov^\textnormal{HT}_G(X)$ has a
	Teichmüller orbit which accumulates on an element of $\Cov^\textnormal{HT}_G(X)$ which
	is connected.
\end{prop}

Here, $S_n$ is the permutation group of $\{1, \dotsc, n\}$ and almost every
means with respect to a natural Borel probability measure on $\Cov^\textnormal{HT}_G(X)$, which
is induced by the counting measure on~$G$,
see~\cite[Section~4.2]{hooper_trevino_19}.

The above result motivates the following terminology. Let $X$ be an infinite translation
surface of finite area, which has non-divergent
Teichmüller orbit in $\mathcal{O}(X)$. Fix an integer $n\ge 2$, and let~$G$ be
any subgroup of $S_n$ which acts transitively on $\{1,\dotsc,n\}$. We call a
connected cover~$Y$ of~$X$ with deck group $G$ \emph{devious} if its Teichmüller orbit in
$\widetilde{\mathcal{O}}_G(X)$ only accumulates on disconnected covers.

\begin{rem}[Translation equivalence vs.\ topological isomorphisms]
	The definitions we just recalled from~\cite{hooper_trevino_19} are with
	respect to translation equivalence. Elsewhere in this article, we have used
	the coarser notion of topological isomorphism. In general, $\TransCov_G(X)$,
	the space of covers up to translation equivalence of a given translation
	surface $X$, with finite deck group~$G$, is a quotient of
	$\Cov^\textnormal{HT}_G(X)$. However, in the case of the Chamanara surface, the two spaces
	are actually isomorphic, see \cref{lem:aut=deck}.
\end{rem}

Having recalled the terminology needed, we now obtain the \hyperlink{cor:devious}{corollary}. 

\begin{proof}[Proof of the \hyperlink{cor:devious}{Corollary to Theorem 3}]
	Since $H = g_{\log{2}} \in \Gamma(X)$, the Chamanara surface $X$ has periodic and hence non-divergent Teichmüller orbit inside $\mathcal{O}(X)$.
	
	In general, if $Y$ is a finite abelian cover of $X$ and $\Gamma(Y)$ has finite index in $\Gamma(X)$, then by \cref{thm:characterization_finite_index}, $Y$ is a fixed point of some power of $H$. This implies that the Teichmüller orbit of~$Y$ is periodic and hence non-divergent.
	In particular, the orbit of $Y$ does not accumulate on disconnected covers by construction, hence it cannot be devious.
\end{proof}

\section{Topology of the covers} \label{sec:topology}

We want to conclude this article with an investigation of the topology of the covers that we studied.
Because of the description of the Chamanara surface in \cref{sec:chamanara}, we can describe a cover of the Chamanara surface of degree $d$ by~$d$ copies of a square where open segments on the boundaries of the squares are identified and the remaining points on the boundary are identified to one or several singularities. From this, we can read that not only the Chamanara surface has infinite genus, also any finite cover of it has infinite genus. Furthermore, as for the Chamanara surface, for any finite cover, its metric completion is compact. Hence the ends correspond one-to-one to the singularities (see~\cite[Proposition~3.10]{randecker_16}) and all ends are accumulated by genus. Therefore, to determine the topological type of a finite cover, we only have to work out the number of singularities.

\pagebreak

\begin{prop}[Topology of covers] \label{thm:topology}
	For every $d\in \mathbb{N}$, there exists a finite abelian cover of the
	Chamanara surface with $d$ ends.

	More specifically, for an abelian group $G$ of order $d$ and $h \in
	\TheSequences_G$, the cover $Y_h$ has $d$ ends if and only if there exists
	an $N\in \mathbb{N}$ such that
	$\sum_{k=-N}^{N-1} (-1)^k h_k = 0$
	and $h_{-n} = h_{-N} = h_{N} = h_{n}$ for all $n \geq N$.
\end{prop}
\cref{thm:topology} is a stronger version of \cref{thmnew:topology} and thus implies it.
Furthermore, for $d$ prime, the cover $Y_h$ can have either one or $d$ ends. Hence, \cref{thm:topology} shows that most abelian covers of degree $d$ have one end and therefore are Loch Ness monsters.

In degree~$2$, we can strengthen \cref{thm:topology} to a characterization of the topological type.
For this, we say that an element $h \in \TheSequences_G$ 
is \emph{bi-eventually equal to $g\in G$} if all but finitely many entries of $h$ are equal to $g$.

\begin{cor}[Characterization of Loch Ness monsters for $d=2$] \label{thm:topology_d=2}
	Let $h \in \TheSequences_{\mathbb{Z}/2\mathbb{Z}}$. Then $Y_h$ is a Jacob's ladder if and only if $h$ is either bi-eventually equal to~$0$ and contains an even number of entries with value $1$ or bi-eventually equal to~$1$ and contains an even number of entries with value $0$ (including $h_0$).
	In all other cases, $Y_h$ is a Loch Ness monster.
\end{cor}

\begin{proof}
	As $Y_h$ is a cover of degree $2$, it can have either one or two
	singularities, that is, can be a Loch Ness monster or a Jacob's ladder. For
	the latter, specializing the condition from \cref{thm:topology} to $G =
	\mathbb{Z}/2\mathbb{Z}$ gives the characterization.
\end{proof}

In particular, if $h$ is periodic, then $Y_h$ is topologically a Loch Ness monster (or disconnected). Hence the covers $Y_h$ studied in \cref{sec:d=2} where $\Gamma(Y_h)$ has finite index in~$\Gamma(X)$ are topologically Loch Ness monsters.

To prepare the proof of \cref{thm:topology}, we first examine the singularities of a cover. For this, recall that the Chamanara surface has exactly one singularity, that is, all points on the boundary of the square which are not contained in an (open) edge are identified. A cover of degree $d$ can be described by $d$ squares whose boundaries are identified. Hence with the coordinates that we fixed in \cref{sec:chamanara}, every point on the boundary of a square which is not contained in an (open) edge is identified with, say, the point $(1,0)$ in the lower right corner of at least one square. Therefore, it is enough to determine which of the~$d$ lower right corners are identified. For this, we consider in the copy labelled with $g\in G$ the sequence~$(x_n^g)$ of points $x_n^g$ with coordinates~$(1 - \sfrac{1}{2^n}, \sfrac{1}{2^n} )$ for every~$n\in \mathbb{N}$.
Note that the distance from~$x_n^g$ to the lower right corner of copy $g$ is $\sfrac{\sqrt{2}}{2^n}$ and to any segment $a_m$ or $b_m$ with $m \geq n$ is at most~$\sfrac{\sqrt{2}}{2^n}$.

Next we use the observation that the lower right corners of the copies $g$ and $g'$ are identified if and only if~$d(x_n^g, x_n^{g'}) \to 0$ for $n \to \infty$.
In particular, to show that two lower right corners are identified, we describe paths from $x_n^g$ to $x_n^{g'}$ and calculate their lengths. These paths have to use the edge identifications and hence depend on the entries of~$h \in \TheSequences_G$.

Every element of $G$ can appear as entries of $h$ either infinitely often,
finitely often, or not at all. In particular, there exists an $N \in \mathbb{N}$
such that the elements of $G$ that appear finitely often, only appear in the
entries $h_{-N}, \ldots, h_{N-1}$. Then every element of $G$ that appears
as~$h_n$ for $n \geq N$ or~$n < -N$ must appear infinitely often and is an
accumulation point of the sequence~$(h_n)$ or the sequence $(h_{-n})$. We
consider the set of paths from $x_n^g$ to $x_n^{g'}$ that cross edges which are labelled with
accumulation points in \cref{lem:sufficient_identification_I} and the set of
paths from $x_n^g$ to $x_n^{g'}$ that cross edges which are labelled with elements that appear only
finitely often in~\cref{lem:sufficient_identification_II}.

\begin{lem}[Sufficient condition for identification of lower right corners I] \label{lem:sufficient_identification_I}
	Let $h \in \TheSequences_G$ and~$g, g' \in G$ be accumulation points of the sequence $(h_n)$ or $(h_{-n})$. Then the lower right corners of the copies $g$ and $g'$ are identified in $Y_h$.

\begin{proof}
	Let $n\in \mathbb{N}$. Then there exist $m, m' \geq n$ such that $h_m = g$ or $h_{-m} = g$ and $h_{m'} = g'$ or~$h_{-m'} = g'$. That means, there exists a path from $x_n^{g}$ crossing the edge $a_{m'}$ or $b_{m'}$ into the copy $g+g'$ and from there crossing the edge $a_{m}$ or $b_m$ (in the reverse direction) into the copy~$g+g' - g = g'$ and there to $x_n^{g'}$.
	Each of the three parts of the path can be chosen to have length at most $\sfrac{\sqrt{2}}{2^n}$, hence the total length of this path is at most $3 \cdot \sfrac{\sqrt{2}}{2^n}$.

	Hence, we have $d(x_n^g, x_n^{g'}) \leq \frac{3 \sqrt{2}}{2^{n}}$ which tends to $0$ as $n \to \infty$.
\end{proof}
\end{lem}

Because the considered coverings are normal, \cref{lem:sufficient_identification_I} also shows that for accumulation points~$g, g'$ of $(h_n)$ or $(h_{-n})$, the lower right corners of the copies $0$ and $g-g'$ as well as $g' - g$ are identified.

\begin{lem}[Sufficient condition for identification of lower right corners II] \label{lem:sufficient_identification_II}
	Let $h \in \TheSequences_G$ and~$N \in \mathbb{N}$ such that for $n \geq N$, every $h_n$ is an accumulation point of $(h_n)$ and every $h_{-n}$ is an accumulation point of $(h_{-n})$. Let $g = \sum_{k=-N}^{N-1} (-1)^k h_k$. Then the lower right corners of the copies~$0$ and $g$ are identified.

\begin{proof}
	We show the identification of the lower right corners of the copies $0$ and $g$ in several~steps.

	Note first that if $n,m \geq N$, then by \cref{lem:sufficient_identification_I}, the lower right corners of the copies~$h_n$ and~$h_m$ are identified. Because the covering is normal, this implies that also the lower right corners of the copies $h_n - h_m$ and $0$ are identified.
	The same holds for $h_{n'}$ and $h_{m'}$ for any $n', m' \geq N$, that is, the lower right corners of the copies $h_{m'} - h_{n'}$ and $0$ are identified. Combining this with the previous identification and using again that the cover is normal yields that the lower right corners of the copies $(h_n - h_m) - (h_{m'} - h_{n'}) = (h_n - h_m) + (h_{n'} - h_{m'})$ and $0$ are identified.
	
	Now let $n \in \mathbb{N}$, $n > N$, and $n$ of the same parity as $N$. Applying the above argument iteratively for the term
	\begin{equation*}
		g' \coloneqq (-1)^{N} h_{N} + \ldots + (-1)^{n+1} h_{n+1}
		,
	\end{equation*}
	that is, for $(-1)^N h_N + (-1)^{N+1} h_{N+1}$ and so on up to $(-1)^n h_n + (-1)^{n+1} h_{n+1}$ yields that the lower right corners of the copies $0$ and $g'$ are identified. As the cover is normal, this further implies that the lower right corners of the copies~$g$ and~$g+g'$ are identified.
	
	Analogously, the lower right corners of the copies $0$ and $-g''$ with
	\begin{equation*}
		g'' \coloneqq (-1)^{-N-1} h_{-N-1} + \cdots + (-1)^{-n} h_{-n}
	\end{equation*}
	are identified as well.
	
	To bring everything together, we consider a  path from $x_n^{-g''}$ to the edge $b_n$ within copy $-g''$, from there successively crossing the edges $b_{n}, \ldots, b_1, a_1, \ldots, a_{n+1}$ into the copy $-g'' + (g'' + g + g')$, and continuing in copy $g + g'$ to~$x_n^{g + g'}$. The paths in copy $-g''$ and in copy $g+g'$ can both be chosen to be of length at most~$\sfrac{\sqrt{2}}{2^n}$. Furthermore, the path from edge $b_n$ in copy $-g''$ to edge $a_{n+1}$ in copy~$g+g'$ is contained in a lift of one of two paths in $X$ in a small neighbourhood of the singularity, see \cref{fig:rotational_components}. Hence, it can be made as short as required, say $2 \cdot \sfrac{\sqrt{2}}{2^n}$.

	\begin{figure}
		\centering
		\begin{tikzpicture}[x=1cm,y=1cm,scale=0.5]
			\newcommand\chamanaraseiteodd{
				\draw (0,0) -- (7.75,0);
				\draw[densely dotted] (7.75,0) -- (8,0);
				\draw[thick] (0,0.6) arc (90:0:0.6);
				\fill (0,0) circle (2pt);
				\fill (4,0) circle (2pt);
				\draw[thick] (5.7,0) arc (180:0:0.3);
				\fill (6,0) circle (2pt);
				\fill (7,0) circle (2pt);
				\draw[thick] (7.35,0) arc (180:0:0.15);
				\fill (7.5,0) circle (2pt);
				\fill (7.75,0) circle (2pt);
			}
			\newcommand\chamanaraseiteeven{
				\draw (0,0) -- (7.75,0);
				\draw[densely dotted] (7.75,0) -- (8,0);
				\fill (0,0) circle (2pt);
				\draw[thick] (3.6,0) arc (180:0:0.4);
				\fill (4,0) circle (2pt);
				\fill (6,0) circle (2pt);
				\draw[thick] (6.75,0) arc (180:0:0.25);
				\fill (7,0) circle (2pt);
				\fill (7.5,0) circle (2pt);
				\draw[thick] (7.65,0) arc (180:0:0.1);
				\fill (7.75,0) circle (2pt);
			}

			\chamanaraseiteodd
			\path (0,0) -- node[below] {$b_1$} (4,0) -- node[below] {$b_2$} (6,0) -- node[below] {$b_3$} (7,0) -- node[below] {$b_4$} (7.5,0);

			\begin{scope}[xscale=-1,yscale=-1,xshift=-8cm,yshift=-8cm]
				\chamanaraseiteeven
				\path (0,0) -- node[above] {$b_1$} (4,0) -- node[above] {$b_2$} (6,0) -- node[above] {$b_3$} (7,0) -- node[above] {$b_4$} (7.5,0);
			\end{scope}

			\begin{scope}[xscale=-1,rotate=90]
				\chamanaraseiteodd
				\path (0,0) -- node[left] {$a_1$} (4,0) -- node[left] {$a_2$} (6,0) -- node[left] {$a_3$} (7,0) -- node[left] {$a_4$} (7.5,0);
			\end{scope}

			\begin{scope}[xscale=-1,rotate=-90,xshift=-8cm,yshift=-8cm]
				\chamanaraseiteeven
				\path (0,0) -- node[right] {$a_1$} (4,0) -- node[right] {$a_2$} (6,0) -- node[right] {$a_3$} (7,0) -- node[right] {$a_4$} (7.5,0);
			\end{scope}

			\begin{scope}[xshift=12cm]
				\chamanaraseiteeven
				\path (0,0) -- node[below] {$b_1$} (4,0) -- node[below] {$b_2$} (6,0) -- node[below] {$b_3$} (7,0) -- node[below] {$b_4$} (7.5,0);

				\begin{scope}[xscale=-1,yscale=-1,xshift=-8cm,yshift=-8cm]
					\chamanaraseiteodd
					\path (0,0) -- node[above] {$b_1$} (4,0) -- node[above] {$b_2$} (6,0) -- node[above] {$b_3$} (7,0) -- node[above] {$b_4$} (7.5,0);
				\end{scope}

				\begin{scope}[xscale=-1,rotate=90]
					\chamanaraseiteeven
					\path (0,0) -- node[left] {$a_1$} (4,0) -- node[left] {$a_2$} (6,0) -- node[left] {$a_3$} (7,0) -- node[left] {$a_4$} (7.5,0);
				\end{scope}

				\begin{scope}[xscale=-1,rotate=-90,xshift=-8cm,yshift=-8cm]
					\chamanaraseiteodd
					\path (0,0) -- node[right] {$a_1$} (4,0) -- node[right] {$a_2$} (6,0) -- node[right] {$a_3$} (7,0) -- node[right] {$a_4$} (7.5,0);
				\end{scope}
			\end{scope}
		\end{tikzpicture}
		\caption{A path in $X$ that crosses the edges $b_{n}, \ldots, b_1, a_1, \ldots, a_{n+1}$ successively (on the left for~$n$ even and on the right for $n$ odd).}
		\label{fig:rotational_components}
	\end{figure}

	Hence, we have $d(x_n^{-g''}, x_n^{g+ g'}) \leq \frac{\sqrt{2}}{2^{n-2}}$ which tends to $0$ as $n \to \infty$.
	This implies that the lower right corners of the copies $-g''$ and $g+g'$ are identified.
	Together with the first observations, we obtain that the lower right corner of copy $0$ is identified with the lower right corner of copy~$-g''$, hence of copy~$g + g'$, and hence of copy $g$.
\end{proof}
\end{lem}

In the next proposition, we show that there are no other ways of identifying lower right corners than the ones from \cref{lem:sufficient_identification_I,lem:sufficient_identification_II} and combinations thereof.

\begin{prop}[Criterion for identification of lower right corners] \label{prop:criterion_identification_corners}
	Let $G$ be a finite abelian group, $h \in \TheSequences_G$, and $N$ as in \cref{lem:sufficient_identification_II}.
	Let $G'$ be the group generated by $\sum_{k=-N}^{N-1} (-1)^{k} h_k$ and by all elements of the form $g-g'$ where $g$ and $g'$ are accumulation points of the sequences $(h_n)$ and~$(h_{-n})$.

	Then for any $g, g' \in G$, the lower right corners of the copies $g$ and $g'$ are identified if and only if~$g-g' \in G'$.

	In particular, the number of ends of $Y_h$ is $[G: G']$.

\begin{proof}
	We show the following equivalent statement: Let $g \in G$. The lower right
	corners of the copies $0$ and $g$ are identified if and only if $g \in G'$.

	``$\Leftarrow$'':
	We assume first that $g$ is a generator (or its inverse) of $G'$. Then we
	have by
	\cref{lem:sufficient_identification_I,lem:sufficient_identification_II} that
	the lower right corners of the copies $0$ and $g$ (or of the copies $0$ and
	$- g$, and hence of $g$ and $0$) are identified.

	If $g$ is the sum of generators (and their inverses) of $G'$, the statement
	follows from the transitivity of the identification and from the fact that
	the covering is normal.

	``$\Rightarrow$'':
	Let $n \geq N+2$. As the lower right corners of the copies $0$ and $g$ are identified, there exists a path from $x_n^0$ to $x_n^g$ in $Y_h$ of length $2 \cdot \sfrac{\sqrt{2}}{2^{n}}$. We consider which edges $a_m$ and~$b_m$ a path of this length can cross.
	
	Note first that the number of edges that are crossed has to be even for the path to stay close to the corresponding lower right corners. Furthermore, the path is shorter than any edge~$a_m$ or~$b_m$ with $m < n-1$ and shorter than the distance from $x_n^0$ and $x_n^g$ to $a_m$ or $b_m$ with~$m < n-1$. Hence if $m < n-1$, an edge $a_m$ or $b_m$ can only be crossed if~$b_{n-1}, \ldots, b_1, a_1, \ldots, a_{n-2}$ and in particular~$b_{N-1}, \ldots, b_1, a_1, \ldots, a_{N}$ are all crossed successively (potentially in reverse order). This corresponds to the setting of \cref{lem:sufficient_identification_II} and a generator $\sum_{k=-N}^{N-1} (-1)^{k} h_k$ (or its inverse).
	
	If the path crosses an edge $a_m$ or $b_m$ for $m \geq N$, it does so an even number of times and hence we are in the setting of \cref{lem:sufficient_identification_I} and with generators of the form $g '- g''$ where~$g'$ and~$g''$ are accumulation points of the sequences $(h_n)$ and $(h_{-n})$.
	
	With this, we can deduce that the path goes from copy~$0$ to copy $g$ where $g$ can be written as a sum of generators of $G'$ and their inverses. This proves~$g \in G'$.
\end{proof}
\end{prop}

With this proposition, we can now prove \cref{thm:topology}.

\begin{proof}[Proof of \cref{thm:topology}]
	Let $h \in \TheSequences_G$. For $Y_h$ to have $d$ ends, the subgroup $G'$ of $G$ from \cref{prop:criterion_identification_corners} needs to be trivial. That is, the sequences $(h_n)$ and $(h_{-n})$ need to be eventually constant with the same value and $\sum_{k=-N}^{N-1} (-1)^{k} h_k = 0$.

	In particular, for every finite abelian group $G$, there exist infinitely many $h\in \TheSequences_G$ that fulfill these conditions. This can be seen by choosing $N \in \mathbb{N}$ and then choosing $h_{-N}, \ldots, h_{N-2}$ arbitrarily, $h_{N-1}$ such that the last condition is fulfilled, and $h_n = 0$ for every $n \leq -N-1$ and~$n\geq N$.
\end{proof}

\begin{figure}[btph]
	\centering
	\includegraphics[width=0.7\linewidth]{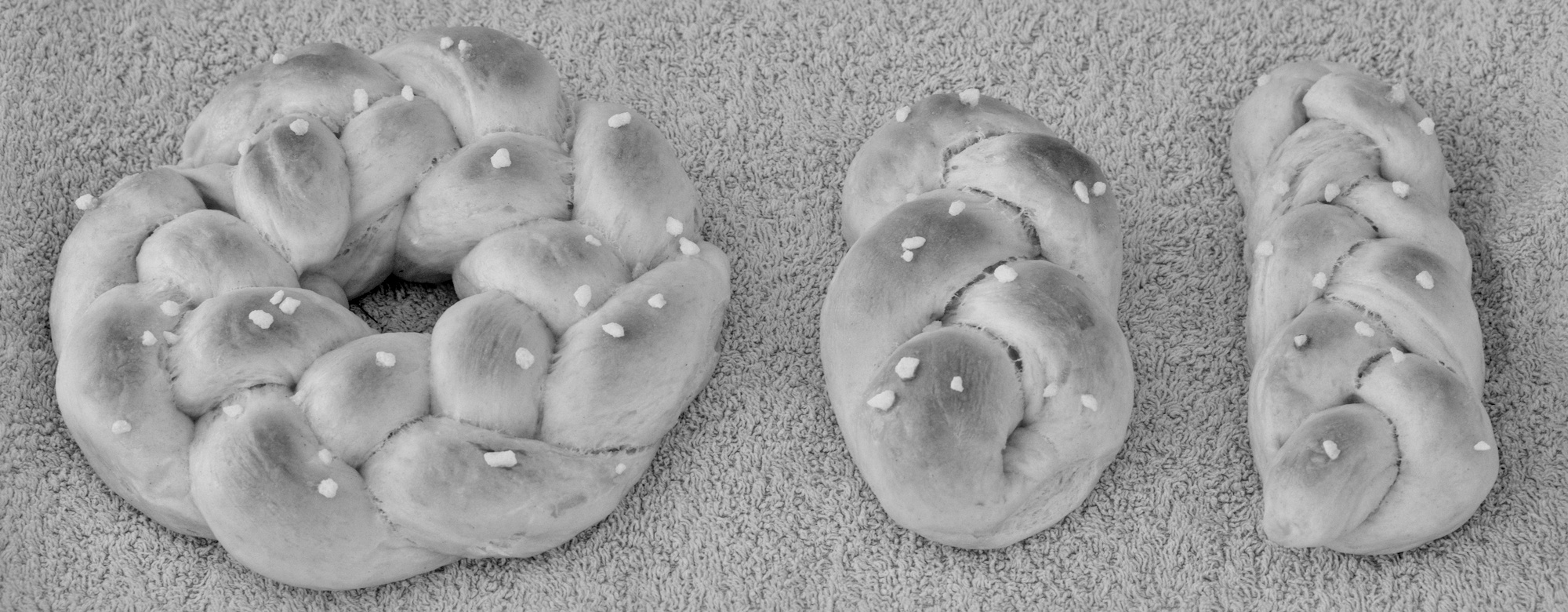}
	\caption{German yeast pastries: a Kranz (left) and two Striezel (middle and right).}
	\label{fig:kranz_striezel}
\end{figure}

\bibliographystyle{amsalpha}
\bibliography{bibliography}

\end{document}